\documentclass[12pt,a4paper,notitlepage,oneside]{article}
\usepackage{latexsym}
\usepackage{array}
\usepackage{amsmath, amsthm,amssymb,amscd,amstext}
\usepackage{amsfonts}
\usepackage{epic}
\usepackage{graphicx}
\usepackage{enumerate}
\input xy
\xyoption{all}
\usepackage[left=3.0cm,right=3.0cm,top=3.0cm,bottom=3.0cm]{geometry}
\newtheoremstyle{thm}{}{}{\itshape}{}{\scshape}{.}{ }{}
\theoremstyle{thm}
\newtheorem{thm}{Theorem}[section]
\newtheorem{lem}[thm]{Lemma}

\newtheorem{prop}[thm]{\normalfont\scshape Proposition}

\newtheorem{cor}[thm]{\normalfont\scshape Corollary}
\renewcommand*{\mod}{\mathrm{mod}\,}
\newcommand{\ind}{\mathrm{ind}\,}
\newcommand{\rad}{\mathrm{rad}}
\renewcommand*{\dim}{\mathrm{dim}}
\newcommand{\add}{\mathrm{add}}
\newcommand{\Tr}{\mathrm{Tr}}

\newcommand{\Hom}{\mathrm{Hom}}
\newcommand{\End}{\mathrm{End}}
\newcommand{\op}{\mathrm{op}}
\newcommand{\D}{\mathrm{D}}
\newcommand{\id}{\mathrm{id}}
\newcommand{\pd}{\mathrm{pd}}
\newcommand{\Ext}{\mathrm{Ext}}
\newcommand{\soc}{\mathrm{soc}}
\newcommand{\res}{\mathrm{res}}
\DeclareMathOperator{\ad}{ad }
\newcommand{\supp}{\mathrm{supp} \hspace{1mm}}
\renewcommand*{\top}{\mathrm{top}}
\renewcommand*{\Im}{\mathrm{Im}}
\newcommand{\ann}{\mathrm{ann}}

\title{{ \Large Selfinjective algebras having a generalized standard family of quasi-tubes
maximally saturated by simple and projective modules}}
\date{}
\author{\normalsize Alicja Jaworska-Pastuszak $^{a,}$\footnote{Corresponding author} , Marta Kwiecie\'n $^{b}$, Andrzej Skowro\'nski $^{a}$}
\begin{document}
\baselineskip=17pt \maketitle \noindent
{\footnotesize $^{a}$
Faculty of Mathematics and Computer Science, Nicolaus
Copernicus University, Chopina 12/18, 87-100 Toru\'n, Poland\\
$^{b}$ Faculty of Mathematics and Computer Science, University of
Warmia and Mazury, ul. S{\l}oneczna 54, 10-710 Olsztyn, Poland}

\footnotetext{E-mail addresses:\\ jaworska@mat.uni.torun.pl (A.
Jaworska-Pastuszak), marta.kwiecien@uwm.edu.pl (M. Kwiecie\'n),
skowron@mat.uni.torun.pl (A. Skowro\'nski)\\
Supported by the Research Grant No. DEC-2011/02/A/ST1/00216 of the
National Science Center Poland.}
\begin{abstract} We give a complete description of finite dimensional
selfinjective algebras over an algebraically closed field whose
Auslander-Reiten quiver admits a  generalized standard family of
quasi-tubes maximally saturated by simple and projective modules.
In particular, we show that these algebras are selfinjective
algebras of strictly canonical type.
\end{abstract}

\section{Introduction and the main results}

Throughout the article $K$ will denote a fixed algebraically
closed field. By an algebra is meant an associative finite
dimensional $K$-algebra with an identity, which we shall assume
(without loss of generality) to be basic and indecomposable. For
an algebra $A$, we denote by $\mod A$ the category of finite
dimensional (over $K$) right $A$-modules, by $\ind A$ its full
subcategory formed by the indecomposable modules, and by $\D :
\mod A \rightarrow \mod A^{\op}$ the standard duality
$\Hom_K(-,K)$. Given a module $M$ in $\mod A$, we denote by $[M]$
the image of $M$ in the Grothendieck group $K_0(A)$ of $A$. Thus
$[M]=[N]$ if and only if the modules $M$ and $N$ have the same
simple composition factors including the multiplicities. An
algebra $A$ is called {\it selfinjective} if $A_A$ is an injective
module, or equivalently, the projective and injective modules in
$\mod A$ coincide. An important class of selfinjective algebras is
formed by the orbit algebras $R/G$, where $R$ is a selfinjective
locally bounded $K$-category and $G$ is an admissible group of
automorphisms of $R$. Then we have a Galois covering $R
\rightarrow R/G$ which frequently  allows us to reduce the
representation theory of $R/G$ to the representation theory of
$R$. In the theory, the selfinjective orbit algebras
$\widehat{B}/G$ given by the repetitive categories $\widehat{B}$
of triangular algebras $B$ and infinite cyclic admissible groups
$G$ of automorphisms of $\widehat{B}$ are of particular interest.
We also note that for the algebras $B$ of finite global dimension,
the stable module category $\underline{\mod} \widehat{B}$ of
$\widehat{B}$ is equivalent (as a triangulated category) to the
derived category $\D^b (\mod B)$ of bounded complexes over $\mod
B$ \cite{Ha}.

An important combinatorial and homological invariant of the module
category $\mod A$ of an algebra $A$ is its Auslander-Reiten
quiver. The Auslander-Reiten quiver $\Gamma_A$ describes the
structure of the quotient category $\mod A / \rad^{\infty}_A$,
where $\rad^{\infty}_A$ is the infinite Jacobson radical of $\mod
A$ (the intersection of all powers $\rad^i_A$, $i \geqslant 1$, of
the Jacobson radical $\rad _A$ of $\mod A$). By a result due to
Auslander \cite{Au}, $A$ is of finite representation type if and
only if $\rad^{\infty}_A=0$ (see also \cite{KS} for an alternative
proof of this result). On the other hand, if $A$ is of infinite
representation type, then $(\rad^{\infty}_A)^2 \neq 0$, by a
result proved in \cite{CMMS}.
 In general, it is important to study
the behaviour of the components of $\Gamma_A$ in the category
$\mod A$. Following \cite{S1}, a family $\mathcal{C}=
(\mathcal{C}_{\lambda})_{\lambda \in \Lambda}$ of components of
$\Gamma_A$ is said to be {\it generalized standard} if
$\rad^{\infty}_A(X,Y)=0$ for all modules $X$ and $Y$ in
$\mathcal{C}$. It was proved in \cite[(2.3)]{S1} that every
generalized standard family $\mathcal{C}$ of components in
$\Gamma_A$ is almost periodic, that is, all but finitely many
$\D\Tr$-orbits in $\mathcal{C}$ are periodic. In particular, for a
selfinjective algebra $A$, every infinite generalized standard
component $\mathcal{C}$ of $\Gamma_A$ is either acyclic with
finitely many $\D\Tr$-orbits or a quasi-tube (the stable part
$\mathcal{C}^s$ of $\mathcal{C}$ is a stable tube
$\mathbb{ZA}_{\infty}/ (\tau^r)$, for some $r \geqslant 1$).

In the paper we are concerned  with the structure of selfinjective
algebras $A$ for which the Auslander-Reiten quiver $\Gamma_A$
admits a generalized standard component. A distinguished class of
such algebras is formed by the selfinjective algebras of finite
representation type. By general theory (see \cite[Section 3]{S4})
these algebras are socle deformations of the orbit algebras
$\widehat{B}/G$, for tilted algebras $B$ of Dynkin type and
infinite cyclic groups $G$ of automorphisms of $\widehat{B}$.
Further, it was proved in \cite{SY1, SY4, SY2} that every
selfinjective algebra $A$ having an acyclic generalized standard
component in $\Gamma_A$ is of the form $\widehat{B}/ G$, for a
tilted algebra $B$ of Euclidean or wild type and an infinite
cyclic group $G$ of automorphisms of $\widehat{B}$. On the other
hand, the description of selfinjective algebras $A$ whose
Auslander-Reiten quiver admits a generalized standard quasi-tube
is an exciting but difficult problem. Namely, every algebra
$\Lambda$ is a quotient algebra of a selfinjective algebra $A$
with $\Gamma_A$ having a generalized standard stable tube (see
\cite{S1}, \cite{S2}). We refer to \cite{BSY}, \cite{Ka},
\cite{KaSY}, \cite{KeSY} for some work on the structure of
selfinjective
algebras having generalized standard families of quasi-tubes.\\

In order to formulate the main result of the paper we need to
present some concept.

Let $A$ be an algebra.
Recall that a {\it smooth} quasi-tube is a quasi-tube whose all nonstable vertices are projective-injective.
For a smooth quasi-tube $\mathcal{T}$
of $\Gamma_A$ we denote by $s(\mathcal{T})$ the number of simple
modules in $\mathcal{T}$, by $p(\mathcal{T})$ the number of
projective modules in $\mathcal{T}$, and by $r(\mathcal{T})$ the
rank of the stable tube $\mathcal{T}^s$. Obviously, if $A$ is selfinjective then each quasi-tube is smooth. Moreover, in this case we know that $s(\mathcal{T})+ p(\mathcal{T}) \leqslant r(\mathcal{T}) -1$ for any quasi-tube $\mathcal{T}$ in $\Gamma_A$ \cite{MS}. A family
$\mathcal{C}=(\mathcal{C}_{\lambda})_{\lambda \in \Lambda}$ of
smooth quasi-tubes in $\Gamma_A$ is said to be {\it maximally saturated
by simple and projective modules} if there exist
two simple right $A$-modules $S$ and $T$ which are not in $\mathcal{C}$
and the following conditions are satisfied:
\begin{enumerate}[\rm (MS1)]
\item $s(\mathcal{C}_{\lambda})+p(\mathcal{C}_{\lambda})=r (\mathcal{C}_{\lambda})-1$ for any ${\lambda} \in \Lambda$;
\item the simple composition factors of indecomposable modules in $\mathcal{C}$
are $S$, $T$, the simple modules in $\mathcal{C}$, and the socles
and tops of indecomposable projective modules in $\mathcal{C}$;
\item $\mathcal{C}=(\mathcal{C}_{\lambda})_{\lambda \in \Lambda}$ consists of
 all quasi-tubes such that $\mathcal{C}_{\lambda}$ admits an indecomposable module $E_{\lambda}$ with $\soc E_{\lambda}=S$ and $\top
 E_{\lambda}=T$.

%\item there exists a multiplicity-free vector $h \in K_0(A)$ such that every quasi-tube $\mathcal{C}_{\lambda}$ in $\mathcal{C}$ admits an indecomposable module $E_{\lambda}$ with $[E_{\lambda}]=h$, $\soc E_{\lambda}$ containing  $S$ and $\top E_{\lambda}$ containing  $T$;
%\item every indecomposable module $X$ in $\Gamma_A$ with $[X]=h$ belongs to $\mathcal{C}$.
\end{enumerate}
In particular, if $p(\mathcal{C}_{\lambda})=0$ for any $\lambda
\in \Lambda$, we say that a family
$\mathcal{C}=(\mathcal{C}_{\lambda})_{\lambda \in \Lambda}$ of
stable tubes is {\it maximally saturated by simple modules}.

The following main result of the paper describes the structure of
all selfinjective algebras whose Auslander-Reiten quiver admits a
generalized standard family of quasi-tubes maximally saturated by
simple and projective modules.

\begin{thm} \label{thm}
Let $A$ be  a basic, connected, finite dimensional selfinjective
algebra over an algebraically closed field $K$. The following
statements are equivalent.
\begin{enumerate}[\rm (i)]
\item $\Gamma_A$ admits a generalized standard family $\mathcal{C}=(\mathcal{C}_{\lambda})_{{\lambda} \in \Lambda}$
of quasi-tubes maximally saturated by simple and projective
modules.
\item $A$ is isomorphic to an orbit algebra $\widehat{B}/G$ of the repetitive category $\widehat{B}$ of a branch extension $B$
of a canonical algebra $C$ with respect to the canonical
$\mathbb{P}_1(K)$-family of stable tubes of $\Gamma_C$ and $G$ is an infinite
cyclic group of automorphisms of $\widehat{B}$ of one of the
forms:
\begin{enumerate}[\rm (a)]
\item $G=(\varphi \nu_{\widehat{B}})$, for $\varphi$ a strictly positive automorphism of $\widehat{B}$,
\item $G=(\varphi \nu_{\widehat{B}})$, for $B$ a canonical algebra  and $\varphi$ a rigid automorphism of $\widehat{B}$,
\end{enumerate}
where $\nu_{\widehat{B}}$ is the Nakayama automorphism of $\widehat{B}$.
\end{enumerate}
\end{thm}

We note that the selfinjective algebras occuring  in the second
statement of the above theorem form a class of selfinjective
algebras of strictly canonical type investigated in \cite{KS1},
\cite{KS2}, \cite{KS3}. In particular, the structure and
homological properties of the Auslander-Reiten quivers of
selfinjective algebras of strictly canonical type were described
in \cite{KS1}.

The following direct consequence of Theorem 1.1 and \cite[Theorem
2]{OTY} provides a characterization of the trivial extensions of the
canonical algebras by their minimal injective cogenerators.

\begin{cor} \label{1.2}
Let $A$ be  a basic, connected, finite dimensional symmetric
algebra over an algebraically closed field $K$. The following
statements are equivalent.
\begin{enumerate}[\rm (i)]
\item $\Gamma_A$ admits a generalized standard family $\mathcal{C}=(\mathcal{C}_{\lambda})_{{\lambda} \in \Lambda}$
of quasi-tubes maximally saturated by simple and projective
modules.
\item $\Gamma_A$ admits a generalized standard family $\mathcal{T}=(\mathcal{T}_{\lambda})_{{\lambda} \in \Lambda}$
of stable tubes maximally saturated by simple modules.
\item $A$ is isomorphic to the trivial extension $B \ltimes D(B)$ of
a canonical algebra $B$.
\end{enumerate}
\end{cor}

The paper is organized in the following way. In Section 2 we
recall the canonical algebras and describe their canonical family
of stable tubes. Section 3 presents quasi-tube enlargements of
algebras. In Section 4 we show the needed facts on repetitive
algebras and their orbit algebras. Section 5 contains results on
selfinjective algebras of strictly canonical type which allow us
to state that the implication (ii) $\Rightarrow$ (i) of Theorem
1.1 is true. In Section 6 we complete the proof of Theorem 1.1.
The final Section 7 is devoted to some examples illustrating the
main results of the paper.

For basic background on the representation theory of algebras
applied in the paper we refer to the books \cite{ASS}, \cite{Ri},
\cite{SS1}, \cite{SS2}, \cite{SY}, \cite{SY0_5}.
%%%%%%%%%%%%%%%%%%%%%%%%%%%%%%%%%%%%%%%%%%%%%%%%%%%%%%%%%%%%%%%%%%%%%%%%%%%%%%%%%%%%%%%%%%%%%%%%%%%%%%%%%%%%%%%%%%%%%%%%%%%%%%%%%%%%%%%%%%%%%%%%%%%%%%%%%%%%%%%%%%%%%%%%%%%%%%%%%%%%%%%%%%%%%%%%%%%%%%%%%
\section{Canonical algebras}

The aim of this section is to introduce the canonical algebras  and show
that they are exactly the  algebras whose Auslander-Reiten quiver contains
a faithful generalized standard family of stable tubes maximally saturated
by simple modules.\\

Throughout the paper for a vertex $x$ in the Gabriel quiver $Q_A$
of an algebra $A$, by $S(x)$, $P(x)$ and $I(x)$ we denote simple,
indecomposable projective and indecomposable injective $A$-module
at vertex $x$, respectively. Moreover, by $M=(M_{a},
M_{\alpha})_{a \in Q_0, \alpha \in Q_1}$, where $Q_0$ is a set of
vertices and $Q_1$ is a set of arrows in $Q_A$,  we denote a
$K$-linear representation of $Q_A$. If $I$ is an admissible ideal
of $KQ_A$, then by representation $M=(M_a, M_{\alpha})$ of $Q_A$
we will mean a representation satisfying the relations in $I$.

Let $m\geq 2$ be an integer number,
$\hbox{\boldmath$p$}=(p_1,\ldots, p_m)$ a sequence of positive
integer numbers and $\hbox{\boldmath$\lambda$}
=(\lambda_1,\ldots,\lambda_m)$ a sequence of pairwise different
elements of the projective line $\mathbb{P}_1(K)=K\cup\{\infty\}$
normalized in such a way that $\lambda_1=\infty$ and
$\lambda_2=0$.
Consider the quiver $\Delta(\hbox{\boldmath$p$})$ of the form\\

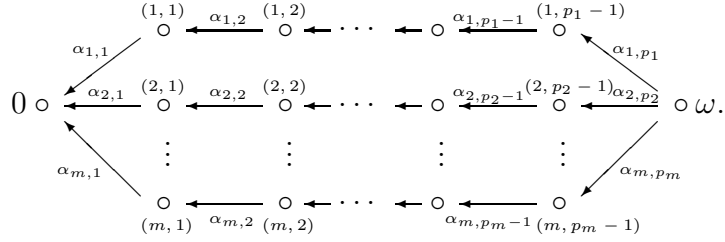
\unitlength1cm
\begin{center}

\begin{picture}(9,2.3)
\multiput(1.6,-0.3)(1.6,0){2}{$\circ$}
\multiput(0,1)(1.6,0){3}{$\circ$}
\multiput(1.6,2)(1.6,0){2}{$\circ$}
\multiput(5.2,-0.3)(1.6,0){2}{$\circ$}
\multiput(5.2,1)(1.6,0){3}{$\circ$}
\multiput(5.2,2)(1.6,0){2}{$\circ$}
\multiput(3,-0.2)(3.6,0){2}{\vector(-1,0){1}}
\multiput(3,1.1)(3.6,0){2}{\vector(-1,0){1}}
\multiput(3,2.1)(3.6,0){2}{\vector(-1,0){1}}
\put(1.4,1.1){\vector(-1,0){1}}
\put(8.2,1.1){\vector(-1,0){1}}
\multiput(5.1,-0.2)(-1.2,0){2}{\vector(-1,0){0.35}}
\multiput(5.1,1.1)(-1.2,0){2}{\vector(-1,0){0.35}}
\multiput(5.1,2.1)(-1.2,0){2}{\vector(-1,0){0.35}}

\put(1.4,2){\vector(-4,-3){1}}
\put(1.4,-0.1){\vector(-1,1){1}}
\put(8.2,1.3){\vector(-4,3){1}}
\put(8.2,0.9){\vector(-1,-1){1}}

\multiput(1.7,0.3)(1.6,0){2}{$\vdots$}
\multiput(5.3,0.3)(1.6,0){2}{$\vdots$} \put(4,-0.2){$\ldots$}
\put(4,1.1){$\ldots$} \put(4,2.1){$\ldots$}
\put(2.3,-0.45){{\tiny{$\alpha_{m,2}$}}}
\put(2.3,1.23){\tiny{$\alpha_{2,2}$}}
\put(2.3,2.23){\tiny{$\alpha_{1,2}$}}
\put(5.4,-0.45){{\tiny{$\alpha_{m,p_m-1}$}}}
\put(5.5,1.23){\tiny{$\alpha_{2,p_2-1}$}}
\put(5.5,2.23){\tiny{$\alpha_{1,p_1-1}$}}
\put(0.7,1.23){\tiny{$\alpha_{2,1}$}}
\put(7.6,1.23){\tiny{$\alpha_{2,p_2}$}}
\put(0.5,1.8){\tiny{$\alpha_{1,1}$}}
\put(0.3,0.2){\tiny{$\alpha_{m,1}$}}
\put(7.6,1.8){\tiny{$\alpha_{1,p_1}$}}
\put(7.7,0.2){\tiny{$\alpha_{m,p_m}$}} \put(-0.3,1){$0$}
\put(8.7,1){$\omega .$} \put(1.4,2.3){\tiny{$(1,1)$}}
\put(3,2.3){\tiny{$(1,2)$}} \put(6.6,2.3){\tiny{$(1,p_1-1)$}}
\put(1.4,1.3){\tiny{$(2,1)$}} \put(3,1.3){\tiny{$(2,2)$}}
\put(6.45,1.3){\tiny{$(2,p_2-1)$}} \put(1.4,-0.54){\tiny{$(m,1)$}}
\put(3,-0.54){\tiny{$(m,2)$}} \put(6.6,-0.54){\tiny{$(m,p_m-1)$}}
\end{picture}
\end{center}
\vspace{0.6cm}

\noindent For $m=2$, $C(\hbox{\boldmath$p$}, \hbox{\boldmath$\lambda$})$ is defined to be the path algebra $K\Delta(\hbox{\boldmath$p$})$ of the quiver $\Delta(\hbox{\boldmath$p$})$ over $K$. For $m\geq 3$, $C(\hbox{\boldmath$p$}, \hbox{\boldmath$\lambda$})$ is defined to be the quotient algebra $K\Delta(\hbox{\boldmath$p$})/I(\hbox{\boldmath$p$}, \hbox{\boldmath$\lambda$})$ of the path algebra $K\Delta(\hbox{\boldmath$p$})$ by the ideal $I(\hbox{\boldmath$p$}, \hbox{\boldmath$\lambda$})$ of $K\Delta(\hbox{\boldmath$p$})$ generated by the elements
$$\alpha_{j,p_j}\ldots \alpha_{j,1}+\alpha_{1,p_1}\ldots \alpha_{1,1}+\lambda_j \alpha_{2,p_2}\ldots \alpha_{2,1}, \;\; \text{where } j\in\{3,\ldots, m\}.$$
Following \cite{Ri}, $C(\hbox{\boldmath$p$}, \hbox{\boldmath$\lambda$})$ is said to be a {\it canonical algebra} of type $(\hbox{\boldmath$p$}, \hbox{\boldmath$\lambda$})$, $\hbox{\boldmath$p$}$ the \textit{weight sequence} of $C(\hbox{\boldmath$p$}, \hbox{\boldmath$\lambda$})$, and $\hbox{\boldmath$\lambda$}$ the \textit{(normalized) parameter sequence} of $C(\hbox{\boldmath$p$}, \hbox{\boldmath$\lambda$})$. It follows from \cite[(3.7)]{Ri} that, for a canonical algebra $C=C(\hbox{\boldmath$p$}, \hbox{\boldmath$\lambda$})$, the Auslander-Reiten quiver $\Gamma_C$ of $C$ is of the form
$$\Gamma_C=\mathcal{P}^C\cup\mathcal{T}^C\cup\mathcal{Q}^C$$
where $\mathcal{P}^C$ is a family of components containing all indecomposable projective $C$-modules (hence the unique simple projective $C$-module $S(0)$ associated with the vertex $0$ of $\Delta(\hbox{\boldmath$p$})$), $\mathcal{Q}^C$ is a family of components containing all indecomposable injective $C$-modules (hence the unique simple injective $C$-module $S(\omega)$ associated with the vertex $\omega$ of $\Delta(\hbox{\boldmath$p$})$), and $\mathcal{T}^C=(\mathcal{T}^C_{\lambda})_{\lambda\in\mathbb{P}_1(K)}$ is a \textit{canonical} $\mathbb{P}_1(K)$-family of pairwise orthogonal standard stable tubes separating  $\mathcal{P}^C$ from $\mathcal{Q}^C$ and containing all simple $C$-modules except $S(0)$ and $S(\omega)$. Moreover, if $r^C_{\lambda}$ denotes the rank of the stable tube $\mathcal{T}^C_{\lambda}$, then $r^C_{\lambda_i}=p_i$, for  any $i\in\{1,\ldots, m\}$, and $r^C_{\lambda}=1$, for $\lambda\in\mathbb{P}_1(K)\setminus \{\lambda_1,\ldots,\lambda_m\}$.\\

Let $C=C(\hbox{\boldmath$p$}, \hbox{\boldmath$\lambda$})$ be a
canonical algebra. We recall the description of modules lying on
the mouths of stable tubes of the canonical
$\mathbb{P}_1(K)$-family
$\mathcal{T}^C=(\mathcal{T}^C_{\lambda})_{\lambda\in\mathbb{P}_1(K)}$
of $\Gamma_C$:
\begin{enumerate}
\renewcommand{\labelenumi}{(\alph{enumi})}
\item For $\lambda=\lambda_1=\infty$, the mouth of $\mathcal{T}^C_{\lambda}=\mathcal{T}^C_{\infty}$ consists of the simple $C$-modules $S(1,1), \ldots, S(1,p_1-1)$ at the vertices $(1,1),\ldots, (1,p_1-1)$ of $\Delta(\hbox{\boldmath$p$})$, if $p_1\geq 2$, and the nonsimple $C$-module $E^{(\infty)}$ of the form

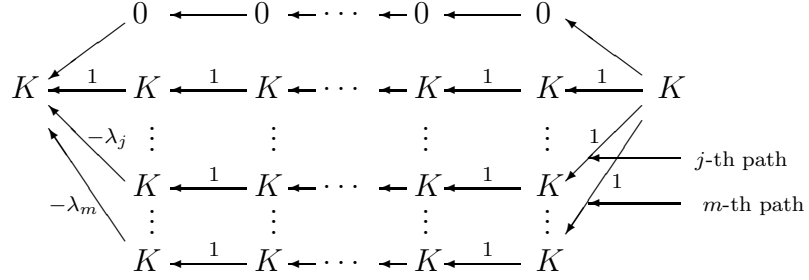
\unitlength1cm
\begin{center}
\begin{picture}(9,2.3)
\multiput(1.5,-0.3)(1.6,0){2}{$K$}
\multiput(-0.1,1)(1.6,0){3}{$K$}
\multiput(1.5,2)(1.6,0){2}{$0$}
\multiput(5.2,-0.3)(1.6,0){2}{$K$}
\multiput(5.2,1)(1.6,0){3}{$K$}
\multiput(5.2,2)(1.6,0){2}{$0$}
\multiput(1.5,-1.3)(1.6,0){2}{$K$}
\multiput(5.2,-1.3)(1.6,0){2}{$K$}

\multiput(3,-0.2)(3.6,0){2}{\vector(-1,0){1}}
\multiput(3,1.1)(3.6,0){2}{\vector(-1,0){1}}
\multiput(3,2.1)(3.6,0){2}{\vector(-1,0){1}}
\put(1.4,1.1){\vector(-1,0){1}}
\put(8.2,1.1){\vector(-1,0){1}}
\multiput(5.1,-0.2)(-1.2,0){2}{\vector(-1,0){0.35}}
\multiput(5.1,1.1)(-1.2,0){2}{\vector(-1,0){0.35}}
\multiput(5.1,2.1)(-1.2,0){2}{\vector(-1,0){0.35}}

\multiput(3,-1.2)(3.6,0){2}{\vector(-1,0){1}}
\multiput(5.1,-1.2)(-1.2,0){2}{\vector(-1,0){0.35}}
\put(1.4,2){\vector(-4,-3){1}}
\put(1.4,-0.1){\vector(-1,1){1}}
\put(8.2,1.3){\vector(-4,3){1}}
\put(8.2,0.9){\vector(-1,-1){1}}
\put(8.2,0.7){\vector(-2,-3){1}}
\put(1.4,-0.9){\vector(-2,3){1}}

\multiput(1.7,-0.8)(1.6,0){2}{$\vdots$}
\multiput(5.3,-0.8)(1.6,0){2}{$\vdots$}
\multiput(1.7,0.3)(1.6,0){2}{$\vdots$}
\multiput(5.3,0.3)(1.6,0){2}{$\vdots$}
\put(4,-0.2){$\ldots$}
\put(4,1.1){$\ldots$}
\put(4,2.1){$\ldots$}
\put(4,-1.2){$\ldots$}

\put(6.1,-1.1){\scriptsize{$1$}}
\put(6.1,-0.1){{\scriptsize{$1$}}}
\put(6.1,1.2){\scriptsize{$1$}}
\put(2.5,-1.1){\scriptsize{$1$}}
\put(2.5,-0.1){{\scriptsize{$1$}}}
\put(2.5,1.2){\scriptsize{$1$}}

\put(0.9,1.2){\scriptsize{$1$}}
\put(7.6,1.2){\scriptsize{$1$}}
\put(0.9,0.4){\scriptsize{$-\lambda_j$}}
\put(7.5,0.4){\scriptsize{$1$}}
\put(0.4,-0.5){\scriptsize{$-\lambda_m$}}
%\put(7.7,-0.5){\scriptsize{$1$}}
\put(7.8,-0.2){\scriptsize{$1$}}

\put(8.7,-0.4){\vector(-1,0){1.2}}
\put(8.7,0.2){\vector(-1,0){1.2}}
\put(9,-0.5){\scriptsize{$m$-th path}}
\put(8.9,0.1){\scriptsize{$j$-th path}}
\end{picture}
\end{center}
\vspace{1.1cm}
with $j\in\{3,\ldots,m\}$;\\

\item For $\lambda=\lambda_2=0$, the mouth of $\mathcal{T}^C_{\lambda}=\mathcal{T}^C_0$ consists of the simple $C$-modules $S(2,1), \ldots, S(2,p_2-1)$ at the vertices $(2,1),\ldots, (2,p_2-1)$ of $\Delta(\hbox{\boldmath$p$})$, if $p_2\geq 2$, and the nonsimple $C$-module $E^{(0)}$ of the form
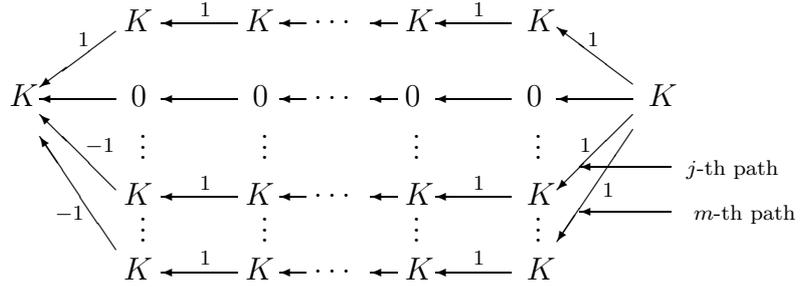
\begin{center}

\begin{picture}(9,2.3)
\multiput(1.5,-0.3)(1.6,0){2}{$K$}
\multiput(1.6,1)(1.6,0){2}{$0$}
\multiput(0,1)(1.6,0){1}{$K$}
\multiput(1.5,2)(1.6,0){2}{$K$}
\multiput(5.2,-0.3)(1.6,0){2}{$K$}
\multiput(5.2,1)(1.6,0){2}{$0$}
\multiput(8.4,1)(1.6,0){1}{$K$}
\multiput(5.2,2)(1.6,0){2}{$K$}
\multiput(1.5,-1.3)(1.6,0){2}{$K$}
\multiput(5.2,-1.3)(1.6,0){2}{$K$}

\multiput(3,-0.2)(3.6,0){2}{\vector(-1,0){1}}
\multiput(3,1.1)(3.6,0){2}{\vector(-1,0){1}}
\multiput(3,2.1)(3.6,0){2}{\vector(-1,0){1}}
\put(1.4,1.1){\vector(-1,0){1}}
\put(8.2,1.1){\vector(-1,0){1}}
\multiput(5.1,-0.2)(-1.2,0){2}{\vector(-1,0){0.35}}
\multiput(5.1,1.1)(-1.2,0){2}{\vector(-1,0){0.35}}
\multiput(5.1,2.1)(-1.2,0){2}{\vector(-1,0){0.35}}

\multiput(3,-1.2)(3.6,0){2}{\vector(-1,0){1}}
\multiput(5.1,-1.2)(-1.2,0){2}{\vector(-1,0){0.35}}
\put(1.4,2){\vector(-4,-3){1}}
\put(1.4,-0.1){\vector(-1,1){1}}
\put(8.2,1.3){\vector(-4,3){1}}
\put(8.2,0.9){\vector(-1,-1){1}}
\put(8.2,0.7){\vector(-2,-3){1}}
\put(1.4,-0.9){\vector(-2,3){1}}

\multiput(1.7,-0.8)(1.6,0){2}{$\vdots$}
\multiput(5.3,-0.8)(1.6,0){2}{$\vdots$}
\multiput(1.7,0.3)(1.6,0){2}{$\vdots$}
\multiput(5.3,0.3)(1.6,0){2}{$\vdots$}
\put(4,-0.2){$\ldots$}
\put(4,1.1){$\ldots$}
\put(4,2.1){$\ldots$}
\put(4,-1.2){$\ldots$}

\put(6.1,2.2){\scriptsize{$1$}}
\put(6.1,-1.1){\scriptsize{$1$}}
\put(6.1,-0.1){{\scriptsize{$1$}}}
\put(2.5,2.2){{\scriptsize{$1$}}}
\put(2.5,-1.1){\scriptsize{$1$}}
\put(2.5,-0.1){{\scriptsize{$1$}}}

\put(0.9,1.8){\scriptsize{$1$}}
\put(7.6,1.8){\scriptsize{$1$}}
\put(1,0.4){\scriptsize{$-1$}}
\put(7.5,0.4){\scriptsize{$1$}}
\put(0.6,-0.5){\scriptsize{$-1$}}
%\put(7.7,-0.5){\scriptsize{$1$}}
\put(7.8,-0.2){\scriptsize{$1$}}

\put(8.7,-0.4){\vector(-1,0){1.2}}
\put(8.7,0.2){\vector(-1,0){1.2}}
\put(9,-0.5){\scriptsize{$m$-th path}}
\put(8.9,0.1){\scriptsize{$j$-th path}}

\end{picture}
\end{center}
\vspace{0.9cm}
with $j\in\{3,\ldots,m\}$;\\

\item For $\lambda=\lambda_j$ with $j\in \{3,\ldots,m\}$, the mouth of $\mathcal{T}^C_{\lambda}$ consists of the simple $C$-modules $S(j,1), \ldots, S(j,p_j-1)$ at the vertices $(j,1),\ldots, (j,p_j-1)$ of $\Delta(\hbox{\boldmath$p$})$, if $p_j\geq 2$, and the nonsimple $C$-module $E^{(\lambda_j)}$ of the form

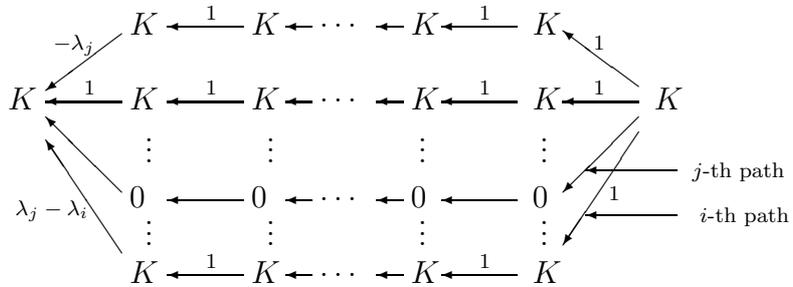
\begin{center}

\begin{picture}(9,2.3)
\multiput(1.5,-0.3)(1.6,0){2}{$0$}
\multiput(-0.1,1)(1.6,0){3}{$K$}
\multiput(1.5,2)(1.6,0){2}{$K$}
\multiput(5.2,-0.3)(1.6,0){2}{$0$}
\multiput(5.2,1)(1.6,0){3}{$K$}
\multiput(5.2,2)(1.6,0){2}{$K$}
\multiput(1.5,-1.3)(1.6,0){2}{$K$}
\multiput(5.2,-1.3)(1.6,0){2}{$K$}

\multiput(3,-0.2)(3.6,0){2}{\vector(-1,0){1}}
\multiput(3,1.1)(3.6,0){2}{\vector(-1,0){1}}
\multiput(3,2.1)(3.6,0){2}{\vector(-1,0){1}}
\put(1.4,1.1){\vector(-1,0){1}}
\put(8.2,1.1){\vector(-1,0){1}}
\multiput(5.1,-0.2)(-1.2,0){2}{\vector(-1,0){0.35}}
\multiput(5.1,1.1)(-1.2,0){2}{\vector(-1,0){0.35}}
\multiput(5.1,2.1)(-1.2,0){2}{\vector(-1,0){0.35}}

\multiput(3,-1.2)(3.6,0){2}{\vector(-1,0){1}}
\multiput(5.1,-1.2)(-1.2,0){2}{\vector(-1,0){0.35}}
\put(1.4,2){\vector(-4,-3){1}}
\put(1.4,-0.1){\vector(-1,1){1}}
\put(8.2,1.3){\vector(-4,3){1}}
\put(8.2,0.9){\vector(-1,-1){1}}
\put(8.2,0.7){\vector(-2,-3){1}}
\put(1.4,-0.9){\vector(-2,3){1}}

\multiput(1.7,-0.8)(1.6,0){2}{$\vdots$}
\multiput(5.3,-0.8)(1.6,0){2}{$\vdots$}
\multiput(1.7,0.3)(1.6,0){2}{$\vdots$}
\multiput(5.3,0.3)(1.6,0){2}{$\vdots$}
\put(4,-0.2){$\ldots$}
\put(4,1.1){$\ldots$}
\put(4,2.1){$\ldots$}
\put(4,-1.2){$\ldots$}

\put(6.1,2.2){\scriptsize{$1$}}
\put(6.1,1.2){\scriptsize{$1$}}
\put(6.1,-1.1){{\scriptsize{$1$}}}
\put(2.5,2.2){{\scriptsize{$1$}}}
\put(2.5,-1.1){\scriptsize{$1$}}
\put(2.5,1.2){{\scriptsize{$1$}}}
\put(0.9,1.2){\scriptsize{$1$}}
\put(7.6,1.2){\scriptsize{$1$}}

\put(0.5,1.8){\scriptsize{$-\lambda_j$}}
\put(7.6,1.8){\scriptsize{$1$}}
%\put(5,0.4){\scriptsize{$-1$}}
%\put(11.5,0.4){\scriptsize{$1$}}
\put(0,-0.4){\scriptsize{$\lambda_j-\lambda_i$}}
\put(7.8,-0.2){\scriptsize{$1$}}

\put(8.7,-0.4){\vector(-1,0){1.2}}
\put(8.7,0.2){\vector(-1,0){1.2}}
\put(9,-0.5){\scriptsize{$i$-th path}}
\put(8.9,0.1){\scriptsize{$j$-th path}}

\end{picture}
\end{center}
\vspace{0.9cm}
for $i\in\{3,\ldots,m\}\setminus\{j\}$;\\

\item For $\lambda\in\mathbb{P}_1(K)\setminus\{\lambda_1,\ldots,\lambda_m\}$, the mouth of $\mathcal{T}^C_{\lambda}$ consists of one nonsimple $C$-module $E^{(\lambda)}$ of the form

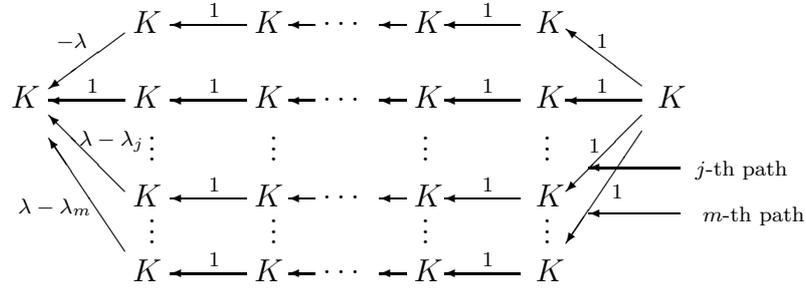
\begin{center}

\begin{picture}(9,2.7)
\multiput(1.5,-0.3)(1.6,0){2}{$K$}
\multiput(-0.1,1)(1.6,0){3}{$K$}
\multiput(1.5,2)(1.6,0){2}{$K$}
\multiput(5.2,-0.3)(1.6,0){2}{$K$}
\multiput(5.2,1)(1.6,0){3}{$K$}
\multiput(5.2,2)(1.6,0){2}{$K$}
\multiput(1.5,-1.3)(1.6,0){2}{$K$}
\multiput(5.2,-1.3)(1.6,0){2}{$K$}

\multiput(3,-0.2)(3.6,0){2}{\vector(-1,0){1}}
\multiput(3,1.1)(3.6,0){2}{\vector(-1,0){1}}
\multiput(3,2.1)(3.6,0){2}{\vector(-1,0){1}}
\put(1.4,1.1){\vector(-1,0){1}}
\put(8.2,1.1){\vector(-1,0){1}}
\multiput(5.1,-0.2)(-1.2,0){2}{\vector(-1,0){0.35}}
\multiput(5.1,1.1)(-1.2,0){2}{\vector(-1,0){0.35}}
\multiput(5.1,2.1)(-1.2,0){2}{\vector(-1,0){0.35}}

\multiput(3,-1.2)(3.6,0){2}{\vector(-1,0){1}}
\multiput(5.1,-1.2)(-1.2,0){2}{\vector(-1,0){0.35}}
\put(1.4,2){\vector(-4,-3){1}}
\put(1.4,-0.1){\vector(-1,1){1}}
\put(8.2,1.3){\vector(-4,3){1}}
\put(8.2,0.9){\vector(-1,-1){1}}
\put(8.2,0.7){\vector(-2,-3){1}}
\put(1.4,-0.9){\vector(-2,3){1}}

\multiput(1.7,-0.8)(1.6,0){2}{$\vdots$}
\multiput(5.3,-0.8)(1.6,0){2}{$\vdots$}
\multiput(1.7,0.3)(1.6,0){2}{$\vdots$}
\multiput(5.3,0.3)(1.6,0){2}{$\vdots$}
\put(4,-0.2){$\ldots$}
\put(4,1.1){$\ldots$}
\put(4,2.1){$\ldots$}
\put(4,-1.2){$\ldots$}

\put(6.1,2.2){\scriptsize{$1$}}
\put(6.1,1.2){\scriptsize{$1$}}
\put(6.1,-1.1){{\scriptsize{$1$}}}
\put(2.5,2.2){{\scriptsize{$1$}}}
\put(2.5,-1.1){\scriptsize{$1$}}
\put(2.5,1.2){{\scriptsize{$1$}}}
\put(0.9,1.2){\scriptsize{$1$}}
\put(7.6,1.2){\scriptsize{$1$}}
\put(6.1,-0.1){\scriptsize{$1$}}
\put(2.5,-0.1){{\scriptsize{$1$}}}

\put(0.5,1.8){\scriptsize{$-\lambda$}}
\put(7.6,1.8){\scriptsize{$1$}}
\put(0.8,0.5){\scriptsize{$\lambda-\lambda_j$}}
\put(7.5,0.4){\scriptsize{$1$}}
\put(0,-0.4){\scriptsize{$\lambda-\lambda_m$}}
%\put(7.7,-0.5){\scriptsize{$1$}}
\put(7.8,-0.2){\scriptsize{$1$}}

\put(8.7,-0.4){\vector(-1,0){1.2}}
\put(8.7,0.2){\vector(-1,0){1.2}}
\put(9,-0.5){\scriptsize{$m$-th path}}
\put(8.9,0.1){\scriptsize{$j$-th path}}

\end{picture}
\end{center}
\vspace{0.9cm}
with $j\in\{3,\ldots,m\}$.\\
\end{enumerate}

The following results describing generalized standard stable tubes
of an Auslander-Reiten quiver were established in
\cite[Corollary 5.3]{S1}.
\begin{prop}
Let $A$ be an algebra and $\Gamma$ a stable tube of $\Gamma_A$.
Then the following statements are equivalent.
\begin{enumerate}[\rm (i)]
\item  $\Gamma$ is generalized standard.
\item $\Gamma$ is standard.
\item The mouth of $\Gamma$ consists of pairwise orthogonal bricks.
\item $\rad_A^{\infty}(X,X)=0$ for any module $X$ in $\Gamma$.
\end{enumerate}
\end{prop}
Recall that an indecomposable $A$-module $X$ is called {\it brick} if its endomorphism algebra $\End_A(X)$ is isomorphic to $K$.\\

We now give the characterization of canonical algebras by means of
a family of stable tubes which are maximally saturated by simple
modules.

\begin{thm} \label{thm1}
Let $B$ be an algebra. Then the following statements are equivalent.
\begin{enumerate}[\rm (i)]
\item  $\Gamma_B$ contains a faithful generalized standard family
  $\mathcal{C}=\{\mathcal{C}_{\lambda}\}_{{\lambda} \in \Lambda}$ of stable tubes maximally saturated by simple modules.
\item $B$ is a canonical algebra.
\end{enumerate}
\end{thm}
\begin{proof}
We proof only that (i) implies (ii), as the other implication is
obvious in view of the above description.

Let $n=rkK_0(B)$ and $S=S(0)$, $T=S(\omega)$ be simple $B$-modules
at the vertices $0$ and $\omega$ of the Gabriel quiver $Q_B$ of
$B$, respectively. Since $\mathcal{C}$ is a generalized standard
family we have that the modules $S$ and $T$ are not isomorphic and
hence $0\neq \omega$. For a module $M \in \mathcal{C}$, by
$l_{\mathcal{C}}(M)$ we shall denote the length of $M$ in the
additive category $\add$ $\mathcal{C}$ of $\mathcal{C}$ in $\mod
B$, that is, the length $l$ of a chain $M=M_0\supset M_1 \supset
...\supset M_l=0$ of submodules of $M$ which belong to
$\mathcal{C}$ and such that $M_{j-1}/M_j$ is a module from the
mouth of $\mathcal{C}_{\lambda}$, for $1 \leq j \leq l$ and some
${\lambda} \in {\Lambda}$. If $X$ is a module from the mouth of
$\mathcal{C}_{\lambda}$, for some ${\lambda} \in {\Lambda}$, by
$X[j]$ we shall denote  a module  $M$ which belongs to infinite
ray starting at $X$ such that $l_{\mathcal{C}}(M)=j$.

We have two cases to consider.

(1) Assume that the family
$\mathcal{C}=\{\mathcal{C}_{\lambda}\}_{{\lambda} \in {\Lambda}}$
consists entirely of homogenous tubes. Since $\mathcal{C}$ is
faithful family maximally saturated by simple modules, we have
that $S$ and $T$ are the only simple $B$-modules. Moreover, by
\cite[Lemma 5.9]{S1} $\pd_A X \leq 1$ and $\id_A X \leq 1$ for any
indecomposable module $X \in \mathcal{C}$. Therefore, the ordinary
quiver $Q_B$ of $B$ does not admit any oriented cycle. Then, by
$(\mathrm{MS3})$, we conclude that $Q_B$ is of the form:
\vspace{0,7cm} \unitlength1cm
\begin{center}
\begin{picture}(3.2,0.9)
\put(-0.1,0.5){0} \put(3.2,0.5){$\omega$}
\put(2.7,1){\vector(-1,0){2.1}} \put(1.65,0.5){$\vdots$}
\put(2.7,0.3){\vector(-1,0){2.1}} \put(2.9,0.5){$\circ$}
\put(0.2,0.5){$\circ$} \put(1.6,1.13){\scriptsize{$\alpha_1$}}
\put(1.6,0.1){\scriptsize{$\alpha_m$}}
\end{picture}
\end{center}
\vspace{0,5cm} for some $m \geq 1$. Assume that
$\alpha_1,...,\alpha_m$ are linearly independent elements of the
$K$-vector space $e_{\omega} B e_0$ and hence $B$ is the path
algebra $KQ_B$. Note that $m \geq 2$ since $B$ is
representation-infinite. If $m \geq 3$, then $B$ is wild and
$\Gamma_B$ does not contain tubes (see \cite[XVIII.1.6]{SS2}).
Hence $B$ is the path algebra $KQ_B$ of the Kronecker quiver:
\begin{center}
\begin{picture}(3.2,0.9)
\put(-0.1,0.5){0}
\put(3.2,0.5){$\omega$}
\put(2.7,0.7){\vector(-1,0){2.1}}
\put(2.7,0.55){\vector(-1,0){2.1}}
\put(2.9,0.5){$\circ$}
\put(0.2,0.5){$\circ$}
\put(1.6,0.83){\scriptsize{$\alpha_1$}}
\put(1.6,0.28){\scriptsize{$\alpha_2$}}
\end{picture}
\end{center}
and $\mathcal{C}=\{\mathcal{C}_{\lambda}\}_{{\lambda} \in \Lambda}$ is a unique
$\mathbb{P}_1(K)$-family of homogenous tubes of $\Gamma_B$
(see \cite[(XI.4.6)]{SS1}).\\

(2) Assume that $\{ \mathcal{C}_1\, ..., \mathcal{C}_m\}$, for
some $m \geqslant 1$, is a complete set of stable tubes of rank at
least two in the family
$\mathcal{C}=\{\mathcal{C}_{\lambda}\}_{{\lambda} \in {\Lambda}}$.
Note that all simple modules from $\mathcal{C}$ lie on mouths of
these tubes. Since $\mathcal{C}$ is faithful, simple modules from
$\mathcal{C}$ together with $S$ and $T$ form the set of all simple
$B$-modules. We mention also that the number $m$ of nonhomogenous
tubes in $\Gamma_B$ is no greater than $n$, by \cite[(2.2)]{S1}.

Let $\mathcal{C}_i$, for some $i \in \{1, ..., m\}$, be a tube of
rank $r(\mathcal{C}_i)=p_{i} \geq 2$. Denote by $S_{(i,1)}, ...,
S_{(i,p_{i}-1)}$ all simple modules which belong to
$\mathcal{C}_{i}$ and by $F_i$ the remaining module from the mouth
of $\mathcal{C}_{i}$. Note that $F_i$ satisfies the condition
(MS3), otherwise an epimorphism $f: E_{\lambda}\rightarrow T$
shall not factorize by a mouth module (see \cite[Lemma
X.2.9]{SS1}).

%Further, $\soc F_i = S$, otherwise, by Proposition
%2.1, there is a nonzero homomorphism from a simple module from
%$\mathcal{C}_j$ to $F_i$, for some $j \neq i$. This contradicts
%the fact that the family $\mathcal{C}$ is generalized standard.
%Analogously, $\top F_i
%=T$.\\

%By $E_i$ we denote an indecomposable module in $\mathcal{C}_i$
%which satisfies the condition (MS3) from the definition of family
%of tubes maximally saturated by simple modules. Since $[E_i]=h$ is
%multiplicity-free we have that $l_{\mathcal{C}}(E_i) \leq p_{i}$.
%But the fact that each tube $\mathcal{C}_j$, for $j \in
%\{1,...,m\}$ contains a module $E_j$ such that $[E_j]=h$ implies
%that $l_{\mathcal{C}}(E_i) \geq p_{i}$. Therefore,
%$l_{\mathcal{C}}(E_i)=p_i$. Moreover, $E_i$ is sincere because
%$\mathcal{C}$ is faithful family of components.\\

We start with describing the Gabriel quiver $Q_B$ of $B$. Since
 $\mathcal{C}$ is generalized standard family of stable tubes, by Proposition
2.1,  we have that
$$\dim_K\Ext^1_B(S',S'')=\dim_K\Hom_B(S'',\tau_B S')=
\begin{cases}
1, &\text{if} \quad \tau_BS' = S''\cr 0, &\text{if} \quad
\tau_BS'\neq S''
\end{cases},$$
for any simple modules $S'$, $S''$ in $\mathcal{C}$. Consider now
the tube $\mathcal{C}_i$. Without loss of generality we assume
that $\tau_BS_{(i,l)} = S_{(i,j)}$ if $l=j+1$, for $j \in \{1,
..., p_i-2\}$ and $\tau_B S_{(i,1)} = F_i$, $\tau_BF_i =
S_{(i,p_i-1)}$. Observe that:
\begin{itemize}
\item $\dim _K\Ext^1_B(T, S_{(i,p_{i}-1)})=\dim_K
\underline{\Hom}_B(\tau_B^{-}S_{(i,p_{i}-1)},
T)=\dim_K\underline{\Hom}_B(F_i, T)=$\\ $=\dim_K\Hom_B(F_i, T)=1$,
since $T=\top F_i$,
\item $\dim _K\Ext^1_B(T, S_{(i,j)})=\dim_K \underline{\Hom}_B(\tau_B^{-}S_{(i,j)},
T)=\dim_K\underline{\Hom}_B(S_{(i,j+1)}, T)=0$, since
$S_{(i,j+1)}, T$ are nonisomorphic simple $B$-modules, for $1 \leq
j<p_{i}-1$,
\item $\dim _K\Ext^1_B(S_{(i,1)},
S)=\dim_K \overline{\Hom}_B(S,\tau_B S_{(i,1)})$ $=\dim_K
\overline{\Hom}_B(S,F_i)=1$, because $S=\soc F_i$,
\item  $\dim _K\Ext^1_B(S_{(i,j)}, S)=\dim_K
\overline{\Hom}_B(S,\tau_B S_{(i,j)})=\dim_K
\overline{\Hom}_B(S,S_{(i,j-1)})=0$, since $S, S_{(i,j-1)}$ are
nonisomorphic simple $B$-modules, for any $1< j\leq p_{i}-1$,
\item $\dim _K\Ext^1_B(S_{(i,j)}, T)=0$, for any $1\leq j\leq
p_{i}-1$, because ${\Hom}_B(T,\tau_B S_{(i,j)})={\Hom}_B(T,S_{(i,
j-1)})=0$, for $j>1$, and ${\Hom}_B(T,\tau_B
S_{(i,1)})={\Hom}_B(T,F_i)={\Hom}_B(\top F_i,F_i)=0$,
\item $\dim
_K\Ext^1_B(S, S_{(i,j)})=0$, for any $1\leq j\leq p_{i}-1$,
because ${\Hom}_B(\tau_B^{-}S_{(i,j)}, S)={\Hom}_B(S_{(i,j+1)},
S)=0$, for $j< p_i-1$, and ${\Hom}_B(\tau_B^{-}S_{(i,p_i-1)},
S)={\Hom}_B(F_i, S)= \Hom_B(F_i, \soc F_i)=0$.
\end{itemize}

Denote by $(i,j)$, for $1 \leq i \leq m$ and $1 \leq j \leq p_i
-1$, the vertex of $Q_B$ for which $S_{(i,j)}$ is the simple
$B$-module at this vertex, that is $S_{(i,j)}=S(i,j)$. Then by
\cite[(III.2.12)]{ASS}, we obtain that $Q_B$  has a subquiver $Q$
of the form \vspace{1cm}

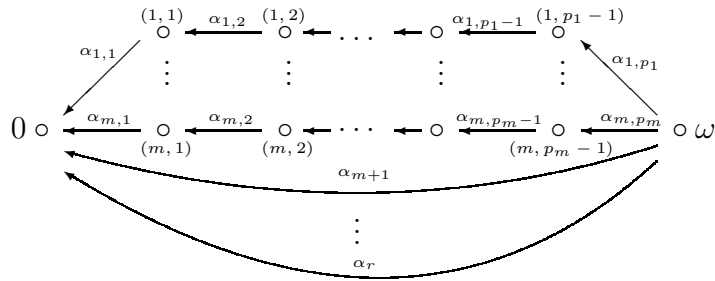
\unitlength1cm
\begin{center}
\begin{picture}(6.2,2.2)
%\multiput(0.6,-0.3)(1.6,0){2}{$\circ$}
\multiput(-1,1)(1.6,0){3}{$\circ$}
\multiput(0.6,2.3)(1.6,0){2}{$\circ$}
%\multiput(4.2,-0.3)(1.6,0){2}{$\circ$}
\multiput(4.2,1)(1.6,0){3}{$\circ$}
\multiput(4.2,2.3)(1.6,0){2}{$\circ$}

%\multiput(2,-0.2)(3.6,0){2}{\vector(-1,0){1}}
\multiput(2,1.1)(3.6,0){2}{\vector(-1,0){1}}
\multiput(2,2.4)(3.6,0){2}{\vector(-1,0){1}}

\put(0.4,1.1){\vector(-1,0){1}}
\put(7.2,1.1){\vector(-1,0){1}}
%\multiput(4.1,-0.2)(-1.2,0){2}{\vector(-1,0){0.35}}
\multiput(4.1,1.1)(-1.2,0){2}{\vector(-1,0){0.35}}
\multiput(4.1,2.4)(-1.2,0){2}{\vector(-1,0){0.35}}

%\put(0.4,2){\vector(-4,-3){1}}
\put(0.4,2.3){\vector(-1,-1){1}}
%\put(7.2,1.3){\vector(-4,3){1}}
\put(7.2,1.3){\vector(-1,1){1}}

\multiput(0.7,1.7)(1.6,0){2}{$\vdots$}
\multiput(4.3,1.7)(1.6,0){2}{$\vdots$}
%\put(3,-0.2){$\ldots$}
\put(3,1.1){$\ldots$}
\put(3,2.3){$\ldots$}
%\put(1.3,-0.45){{\tiny{$\alpha_{m,2}$}}}
\put(1.3,1.23){\tiny{$\alpha_{m,2}$}}
\put(1.3,2.53){\tiny{$\alpha_{1,2}$}}
%\put(4.4,-0.45){{\tiny{$\alpha_{m,p_m-1}$}}}
\put(4.55,1.23){\tiny{$\alpha_{m,p_m-1}$}}
\put(4.5,2.53){\tiny{$\alpha_{1,p_1-1}$}}
\put(-0.3,1.23){\tiny{$\alpha_{m,1}$}}
\put(6.45,1.23){\tiny{$\alpha_{m,p_m}$}}
\put(-0.4,2.1){\tiny{$\alpha_{1,1}$}}
\put(6.6,2){\tiny{$\alpha_{1,p_1}$}}

\put(-1.3,1){$0$}
\put(7.7,1){$\omega$}
\put(0.4,2.57){\tiny{$(1,1)$}}
\put(2,2.57){\tiny{$(1,2)$}}
\put(5.6,2.57){\tiny{$(1,p_1-1)$}}
\put(0.4,0.8){\tiny{$(m,1)$}}
\put(2,0.8){\tiny{$(m,2)$}}
\put(5.25,0.8){\tiny{$(m,p_m-1)$}}

\put(-0.5,0.8){\vector(-3,1){0.1}}
\bezier{1000}(-0.5,0.8)(3.5,-0.3)(7.2,0.9)
\put(-0.5,0.5){\vector(-3,2){0.1}}
\bezier{700}(-0.5,0.5)(4,-2.3)(7.2,0.7)
\put(3, 0.5){\tiny{$\alpha_{m+1}$}}
\put(3.2,-0.4){$\vdots$}
\put(3.2,-0.7){\tiny{$\alpha_{r}$}}

\end{picture}
\end{center}
\vspace{0,5cm}

\vspace{0,6cm} \noindent where $r\geq m$,  there are no other
vertices in $Q_B$, no other arrows starting at or ending in
vertices $(i,j)$ and no other arrows starting at $\omega$ and
ending in $0$. Since $\mathcal{C}$ form a hereditary family of
modules in $\mod B$ (see \cite[Lemma 5.9]{S1}), then any path
$\alpha_{i,p_i}\ldots\alpha_{i,1}$, for $1 \leq i \leq m$, is a
nonzero element of $B$, and $Q$ does not admit any oriented cycle.
Therefore, we conclude that $Q_B=Q$.

Consider now the algebra $B'=eBe$, where $e=e_0+e_{\omega}$. There
is the canonical restriction functor
\[\res_e: \mod B \rightarrow \mod{B'}\]
which assigns to a module $M$ in $\mod B$ the module
$\res_e(M)=Me$ in $\mod B'$ and to a homomorphism $f: M
\rightarrow N$ in $\mod B$ its restriction $\res_e(f): \res_e(M)
\rightarrow \res_e(N)$ to $Me$. Note that $Q_{B'}$ is an enlarged
Kronecker quiver with $r$ arrows
\vspace{0,5cm}
\begin{center}
\begin{picture}(3.2,1)
\put(-0.1,0.5){0}
\put(3.2,0.5){$\omega$ .}
\put(2.7,0.8){\vector(-1,0){2.1}}
\put(2.7,0.35){\vector(-1,0){2.1}}
\put(2.9,0.5){$\circ$}
\put(0.2,0.5){$\circ$}
\put(1.6,0.93){\scriptsize{$\alpha_1$}}
\put(1.6,0.08){\scriptsize{$\alpha_r$}}
\put(1.7,0.4){$\vdots$}
\end{picture}
\end{center}
%\vspace{0,5cm}
Now we apply the functor $\res _e$ to the additive subcategory
$\add (\mathcal{C}_{i})$ of $\mod B$. Observe that
$\res_e(S_{(i,1)})=\res_e(S_{(i,2)})=...=\res_e(S_{(i,p_{i}-1)})=0$.
Moreover, since $F_i=\tau_BS_{(i,1)}$, we have
$\res_e(F_{i})=Z=(Z_x,Z_{\alpha})_{x,\alpha}$, where $Z_x=K$ for
the vertices $0, \omega$, and $Z_{\alpha}=0$ only for
$\alpha=\alpha_i$. Clearly, $\res_e(F_{i})$ is an indecomposable
$B'$-module. We shall now fix $E_{i}=S_{(i,1)}[p_{i}]$. Let
$M_{i}=F_{i}[p_{i}+1]$. Since $\tau_B^-F_{i}=S_{(i,1)}$, there
exists an exact sequence of $B$-modules
\[\xymatrix{ 0 \ar[r] & F_{i} \ar[r]^(0.5){f} & M_{i} \ar[r]^(0.4)g & E_i\ar[r] &0}.\]
Then we have an exact sequence of $B'$-modules
\[\xymatrix{0 \ar[r] & \res_e(F_{i}) \ar[r]^(0.5){\res_e(f)} & \res_e(M_{i})
\ar[r]^(0.5){\res_e(g)} & \res_e(E_i)\ar[r] &0}\] because the
restriction functor $\res_e$ is exact (see \cite[Theorem
I.6.8]{ASS}). Observe that $\res_e(E_{i})=\res_e(F_{i})$.
Moreover, $W=\res_e (M_{i})$ is indecomposable, since $M_{i}$ is
indecomposable and, for $j \in \{1,...,m\}$, $W_{\alpha_j}: K^2
\rightarrow K^2$ is given by $W_{\alpha_j}=(M_{i})_{\varrho}$,
where $\varrho=\alpha_{j,
p_j}\alpha_{j,p_{j-1}}...\alpha_{j,2}\alpha_{j,1}$.

We are now in  the position to show that $\res_e(f)$ is a left
almost split homomorphism in $\mod B'$. Since $\res_e(M_{i})$ is
indecomposable and $\res_e(g)\neq 0$, $\res_e(f)$ is not a
section. Let $u: \res_e(F_{i})\rightarrow U$, for some
indecomposable $B'$-module $U$, be a non-zero homomorphism which
is not a section. Since $\End_{B'}(\res_e(F_{i}))\cong K$, it
follows that $U \not\cong \res_e(F_{i})$. Invoking the extension
functor $L_e: \mod B' \rightarrow \mod B$,
$L_e(-)=\Hom_{B'}(Be,-)$, which is right adjoint to $\res_e$, we
obtain that there exists a homomorphism $v: F_{i} \rightarrow
L_e(U)$ of $B$-modules such that $u =\res_e(v)$. The functor $L_e$
preserves indecomposability of modules, thus $L_e(U)$ is
indecomposable. Moreover, there is a functorial isomorphism
$\res_e L_e \cong 1_{\mod B}$ (see \cite[Theorem I.6.8]{ASS}).
Hence, $v$ is not a section. We claim that $L_e(U)$ is not of the
form $F_{i}[j]$ for $j \in \{1,...,p_{i}\}$. Observe that
$\res_e(F_{i}[j])=\res_e(F_{i})$ for any $j \in \{1,...,p_{i}\}$.
Then the claim follows from the facts that $\res_eL_e(U)\cong U$
and $U\not\cong \res_e(F_{i})$. Therefore, $v: F_{i} \rightarrow
L_e(U)$ is a composition $wf_{p_{i}}...f_1$ for some homomorphism
$w: M_{i} \rightarrow L_e(U)$ and irreducible homomorphisms $f_j:
F_{i}[j] \rightarrow F_{i}[j+1]$ for $j \in \{1,...,p_{i}\}$ (see
\cite[Lemma IV.5.1]{ASS}). Thus $v$ factorizes through
$M_{i}=F_{i}[p_i+1]$. Invoking now the restriction functor
$\res_e$, we conclude that $u=\res_e(v)$ factorizes through
$\res_e(M_{i})$. It shows that $u$ is left almost split  and, by
\cite[Theorem IV.1.13]{ASS},
\[\xymatrix{0 \ar[r] & \res_e(F_{i}) \ar[r]^{\res_e(f)} & \res_e(M_{i}) \ar[r]^(0.5){\res_e(g)} & \res_e(E_i)\ar[r] &0}\]
is an almost split sequence in $\mod B'$. Consequently, the image
of $\mathcal{C}_{i}$ by the  functor $\res_e$ is a homogenous tube
of $\Gamma_{B'}$. Again, by \cite[(XVIII.1.6)]{SS2}, we conclude
that $r=2$ since $B'$ is a tame algebra. Hence $\dim_K
e_{\omega}Be_0=2$.

%Let us denote by $\varrho_i$ the path
%$\alpha_{i,p_i}...\alpha_{i,1}$, $1 \leq i \leq m$. Suppose now
%that, for some $i$, $\varrho_i=0$ as an element of algebra $B$.
%Then there exist $1 \leq l \leq p_i$ and $0 \leq k \leq p_i-l $
%such that $\alpha_{i,k+l}...\alpha_{i,l+1}\alpha_{i,l}=0$. Assume
%$l$ is maximal and $k$ is minimal with this property. Consider
%$\mathcal{C}_i \in \mathcal{C}$ which contains the simple module
%$S((i, l))$. Note that $S((i,l))[k]= \rad P((i,l+k))$ for $l+k
%\neq p_i$ and if $l+k=p_i$ then $S((i,l))[k]$ is a direct summand
%of $\rad P(\omega)$. We obtain a contradiction with the stability
%of $\mathcal{C}_i$. Hence $\varrho_i$ is a non-zero element of
%$B$, for any $1 \leq i \leq m$.

Let us denote by $\varrho_i$ the path
$\alpha_{i,p_i}...\alpha_{i,1}$ if $1 \leq i \leq m$, and the
arrow $\alpha_i$ if $m+1 \leq i \leq r$. Assume now that $\lambda
\varrho_i=\varrho_j$ for some $j \neq i$, $1 \leq j,i \leq r$ and
non-zero $\lambda \in K$. Fix a stable tube $\mathcal{C}_i \in
\mathcal{C}$ which contains the simple modules $S(x)$ at the
vertices $x \in \{(i,1), (i,2),...,(i,p_i-1)\}$. For simplicity,
denote the unique non-simple module $F_i$ which lies on the mouth
of $\mathcal{C}_i$ by $F$. Observe that
$F_{\alpha_{i,1}}=...=F_{\alpha_{i,p_i}}=0$ and hence
$F_{\varrho_i}=F_{\alpha_{i,p_i}...\alpha_{i,1}}=0$. Then
$F_{\varrho_j}=F_{\lambda\varrho_i}=0$ which is impossible since
$F=\tau_B S_{(i,1)}$.

%and since $F_y$, for $y \in \{(j,1), (j,2),...,(j,p_j-1)\}$, is
%one-dimensional $K$-vector space, there exists exactly one $k$,
%$1\leq k \leq p_j$, such that $F_{\alpha_{j,k}}=0$. Clearly, if $k
%\neq p_j$, then $\soc F =S((j,k)) \oplus S$. Since $S((j,k))
%\not\in \mathcal{C}_i$, then the inclusion homomorphism
%$f:S((j,k)) \rightarrow F$ belongs to $\rad^{\infty}_B$. If
%$k=p_j$, then $\top F = T \oplus S((j,p_j-1))$ and
%$\rad^{\infty}_B(F, S((j,p_j-1))) \neq 0$, since $S((j,p_j-1))
%\not \in \mathcal{C}_i$. In both cases we get a contradiction with
%the fact that $\mathcal{C}$ is generalized standard.

Without loss of generality we may now assume that $\varrho_1,
\varrho_2$ form a basis of $e_{\omega}Be_0$. From the above
consideration we conclude that the equations which describe
$\varrho_3,...,\varrho_m, \alpha_{m+1}, ..., \alpha_r$ define the
set $\Omega$ of generic relations (in the sense of \cite{Ri}) in
$KQ_B$. Then, by \cite[(3.7)]{Ri}, $B \cong KQ_B / \langle
\Omega\rangle$ is a canonical algebra and $\mathcal{C}$ is a
canonical separating $\mathbb{P}_1(K)$-family of stable tubes of
$\Gamma_{B}$. Note that $B$ is a canonical algebra of type
({\boldmath$p$},{\boldmath$\lambda$}), where the weight sequence
{\boldmath$p$} contains a subsequence $(p_1,...,p_m)$ and the
parameter sequence {\boldmath$\lambda$} is determined by relations
in $\Omega$. Following \cite[Lemma 1.1]{KS1} (see also
\cite[(3.7)]{Ri}), if $m=r \geq 3$, we may assume that $p_i \geq
2$ for each $i \in \{1,...,m\}$.
%
%If $m=2$, $Q_B$ admits an arrow of source $\omega$
%and target $0$ if and only if  $\mathcal{C}$ contains exactly one
%non-homogenous tube.
\end{proof}

%%%%%%%%%%%%%%%%%%%%%%%%%%%%%%%%%%%%%%%%%%%%%%%%%%%%%%%%%%%%%%%%%%%%%%%%%%%%%%%%%%%%%%%%%%%%%%%%%%%%%%%%%%%%%%%%%%%%%%%%%%%%%%%%%%%%%%%%%%%%%%%%%%%%%%%%%%%%%%%%%%%%%%%%%%%%%%%%%%%%%%%%%%%%%%%%%%%%%%%%%
\section{Quasi-tube enlargements of algebras}

In this section we introduce quasi-tube enlargements of algebras which are essential in the paper.

A connected translation quiver $\Gamma$ is said to be a {\it
quasi-tube} if $\Gamma$ can be obtained from a stable tube by an
iterated application of admissible operations $(ad  1), (ad  2)$,
$(\ad 1 ^{\ast})$, or $(\ad  2^{\ast})$ (see \cite[Section
2]{AS1}, \cite[Section 2]{AST} for details). The following
proposition provides a characterization of quasi-tubes in
Auslander-Reiten quivers of selfinjective algebras  (\cite{L},
\cite{MS1}, \cite{Z}).
\begin{prop}
Let $A$ be a selfinjective algebra and $\Gamma$ a connected component of $\Gamma_A$. The following statements are equivalent.
\begin{enumerate}[\rm (i)]
\item  $\Gamma$ is a quasi-tube.
\item  The stable part $\Gamma^s$ of $\Gamma$ is a stable tube.
\item  $\Gamma$ contains an oriented cycle.
\end{enumerate}
\end{prop}
Recall that the stable part $\Gamma^s$ of $\Gamma$ is obtained
from $\Gamma$ by removing the projective-injective modules and the
arrows attached to them.

Let $A$ be an algebra and $X$ a module in $\mbox{mod}\,A$. The \textit{one-point extension} of $A$ by $X$ is the $2\times 2$-matrix algebra
$$A[X]=\left[\begin{matrix}A&0\cr _KX_A&K\end{matrix}\right]=\left\{\left(\begin{matrix}a&0\cr x&\lambda\end{matrix}\right);\,a\in A,\, \lambda\in K,\,x\in X\right\}$$
with the usual addition of matrices and the multiplication induced from the canonical $K$-$A$-bimodule structure $_KX_A$ of $X$.  Dually, the \textit{one-point coextension} of $A$ by $X$ is the $2\times 2$-matrix algebra
$$[X]A=\left[\begin{matrix}K&0\cr _AD(X)_K&A\end{matrix}\right]=\left\{\left(\begin{matrix}\lambda&0\cr f&a\end{matrix}\right);\,a\in A,\, \lambda\in K,\,f \in D(X)\right\}$$
with the usual addition of matrices and the multiplication induced from the canonical $A$-$K$-bimodule structure $_AD(X)_K$ of $D(X)=\mbox{Hom}_K(_KX_A,K)$.

For an algebra $A$ let $\Gamma$ be a generalized standard
component of $\Gamma_A$. For each indecomposable module $X$ in
$\Gamma$ which is a pivot of an admissible operation of type $(\ad
1), (\ad 2), (\ad  1^{\ast})$, or $(\ad  2^{\ast})$, we shall
define the corresponding admissible operation on $A$ in such a way
that the modified translation quiver $\Gamma'$ is a component of
the Auslander-Reiten quiver $\Gamma_{A'}$ of the modified algebra
$A'$ (see \cite{AS1}, \cite{AST}). Since $\Gamma$ is generalized
standard, such a pivot $X$ is necessarily a brick. Suppose that
$X$ is the pivot of an admissible operation of type $(\ad  1)$ and
that $t$ is a positive integer. Denote $H=H_t$ the full $t \times
t$ upper triangular matrix algebra over $K$ and by $Y$ the unique
indecomposable projective-injective $H$-module, which  we consider
as a $K$-$H$-bimodule. Then $A'=(A \times H)[X \oplus Y]$ is the
required modified algebra. If $X$ is the pivot of an admissible
operation of type $(\ad  2)$, then the modified algebra $A'$ is
defined to be $A'=A[X]$. Dually, invoking the one-point
coextensions, one defines the modified algebra $A'$, if $X$ is a
pivot of an admissible operation of type $(\ad  1^{\ast})$ or
$(\ad 2^{\ast})$. Then the following fact mentioned above holds
(see \cite[Section 2]{AS1}).
\begin{lem}
The modified translation quiver $\Gamma'$ of $\Gamma$ is a component of $\Gamma_{A'}$.
\end{lem}
Let now $C$ be an algebra and $\mathcal{T}$ a generalized standard
family of stable tubes in $\Gamma_C$. Following \cite{AST}, an
algebra $B$ is said to be a {\it quasi-tube enlargement} of $C$
using modules from $\mathcal{T}$ if there is a finite sequence of
algebras $A_0=C, A_1,...,A_m=B$ such that, for each $0 \leq j <
m$, $A_{j+1}$ is obtained from $A_j$ by an admissible operation of
type $(\ad  1), (\ad  2), (\ad  1^{\ast})$, or $(\ad  2^{\ast})$,
with pivot either in a stable tube of $\mathcal{T}$ or in a
quasi-tube of $\Gamma_{A_j}$ obtained from a stable tube of
$\mathcal{T}$ by means of the sequence of admissible operations
(of types $(\ad  1), (\ad  2), (\ad  1^{\ast}), (\ad  2^{\ast})$)
done so far. We note that a \textit{tubular extension }
(respectively, \textit{tubular coextension}) of $C$ (in the sense
of \cite{Ri}), using modules from $\mathcal{T}$, is just an
enlargement of $C$ invoking only admissible operations of type
$(\ad  1)$ (respectively, of type $(\ad  1^{\ast})$).
\begin{prop} \label{prop3.?}
Let  $B$ be a quasi-tube enlargement of an algebra $C$ using modules from a generalized standard family $\mathcal{T}$ of stable tubes of $\Gamma_C$, and $\mathcal{C}$ the family of components of $\Gamma_B$ obtained from $\mathcal{T}$ by means of admissible
operations leading from $C$ to $B$. Then $\mathcal{C}$ is a generalized standard family of quasi-tubes of $\Gamma_B$.
\end{prop}
For the proof of the above proposition we refer to \cite[Lemma 2.2, 2.3]{AS1} and \cite[Theorem C]{MS2}. We end this section with the following consequence for a canonical algebra $C$ \cite[Theorem XV.3.9]{SS2}.
\begin{prop} \label{prop A}
Let $C$ be a canonical algebra and $\mathcal{T}^C$ the canonical $\mathbb{P}_1(K)$-family of standard stable tubes of $\Gamma_C$. For an algebra $A$ the following equivalences hold.
\begin{enumerate}[\rm (i)]
\item $A$ is a $\mathcal{T}^C$-tubular extension of $C$ if and only if $A$ is a branch $\mathcal{T}^C$-extension of $C$.
\item $A$ is a $\mathcal{T}^C$-tubular coextension of $C$ if and only if $A$ is a branch $\mathcal{T}^C$-coextension of $C$.
\end{enumerate}
\end{prop}

Recall from \cite[Section 2]{AST}, \cite[Section 4]{Ri}  that for a branch extension $B$  of a
canonical algebra $C$ the Auslander-Reiten quiver $\Gamma_B$ has a
disjoint union decomposition
$$\Gamma_B=\mathcal{P}^B\vee\mathcal{T}^B\vee\mathcal{Q}^B,$$
where $\mathcal{P}^B=\mathcal{P}^C$ is a family of components consisting
of $C$-modules and containing all indecomposable projective $C$-modules,
$\mathcal{Q}^B$ is a family of components containing all indecomposable
injective $B$-modules but no projective $B$-module, and $\mathcal{T}^B$
is a $\mathbb{P}_1(K)$-family $(\mathcal{T}^B_{\lambda})_{\lambda\in\mathbb{P}_1(K)}$
of pairwise orthogonal standard ray tubes separating $\mathcal{P}^B$
from $\mathcal{Q}^C$.
Respectively, for branch coextension $B$,  $\mathcal{P}^B$ is a family of
components containing all indecomposable projective $B$-modules but no
injective $B$-modules, $\mathcal{Q}^B=\mathcal{Q}^C$ is a family of
components consisting of $C$-modules and containing all indecomposable
injective $C$-modules, and $\mathcal{T}^B$ is a $\mathbb{P}_1(K)$-family
$(\mathcal{T}^B_{\lambda})_{\lambda\in\mathbb{P}_1(K)}$ of pairwise orthogonal
standard coray tubes separating $\mathcal{P}^B$ from $\mathcal{Q}^C$.

%%%%%%%%%%%%%%%%%%%%%%%%%%%%%%%%%%%%%%%%%%%%%%%%%%%%%%%%%%%%%%%%%%%%%%%%%%%%%%%%%%%%%%%%%%%%%%%%%%%%%%%%%%%%%%%%%%%%%%%%%%%%%%%%%%%%%%%%%%%%%%%%%%%%%%%%%%%%%%%%%%%%%%%%%%%%%%%%%%%%%%%%%%%%%%%%%%%%%%%%%
\section{Selfinjective orbit algebras}

In this section we recall needed background on selfinjective orbit algebras.\\

Let $B$ be an algebra and $\mathcal{E}_B =\{e_i\,|1\leq i\leq n\}$
be a fixed set of orthogonal  primitive idempotents of $B$ with
$1_B=e_1+\ldots+e_n$. Then the {\it repetitive category}
$\widehat{B}$ of $B$ is the category with
$\widehat{\mathcal{E}}_B=\{e_{m,i}\,|m\in\mathbb{Z},\, 1\leq i\leq
n\}$ as a set of objects of $\widehat{B}$ and the morphism spaces
defined by
$$\widehat{B}(e_{m,i}, e_{r,j})=\left\{\begin{array}{ccl} e_jBe_i, &  &  r=m \\D(e_iBe_j), & & r=m+1\\0, &  & \mbox{otherwise}\end{array}\right.$$
and the composition of morphisms given by the multiplication in
$B$ and the canonical $B$-$B$-bimodule structure of
$D(B)=\mbox{Hom}_K(B,K)$. For each $m \in \mathbb{Z}$, we denote
by $B_m$ the full subcategory of $\widehat{B}$ given by the
objects $e_{m, i}$ for all $i \in \{1,...,n\}$. Observe that
$\widehat{B}$ is a selfinjective locally bounded $K$-category. An
automorphism $\varphi$ of $\widehat{B}$ is said to be
\begin{itemize}
\item \textit{positive} if, for each pair $(m,i)\in \mathbb{Z}\times \{1,\ldots,n\}$, we have $\varphi(e_{m,i})=e_{p,j}$ for some $p\geq m$ and some $j\in\{1,\ldots,n\}$;
\item \textit{rigid} if, for each pair $(m,i)\in \mathbb{Z}\times \{1,\ldots,n\}$, we have $\varphi(e_{m,i})=e_{m,j}$ for some $j\in\{1,\ldots,n\}$;
\item \textit{strictly positive} if it is positive but not rigid.
\end{itemize}
An important role is played by the {\it Nakayama automorphism}
$\nu_{\widehat{B}}$ of $\widehat{B}$ which is defined by
$$\nu_{\widehat{B}}(e_{m,i})=e_{m+1,i}\;, \quad \text{for all} \quad (m,i) \in \mathbb{Z} \times \{1, ...,n\}.$$

\noindent Note that the Nakayama automorphism $\nu_{\widehat{B}}$
is a strictly positive automorphism of $\widehat{B}$. A group $G$
of automorphisms of $\widehat{B}$ is said to be {\it admissible}
if it acts freely on the set $\widehat{\mathcal{E}}_B$ and has
finitely many orbits.

Let $B$ be an algebra and $G$ be an admissible group of
automorphisms of $\widehat{B}$. Following Gabriel \cite{G}, we may
consider the finite orbit $K$-category $\widehat{B}/G$ defined as
follows. The objects of $\widehat{B}/G$ are the elements $a=Gx$ of
the set $\widehat{\mathcal{E}}_B/G$ of $G$-orbits in
$\widehat{\mathcal{E}}_B$ and the morphism spaces are given by
{\small{$$(\widehat{B}/G)(a,b)=\left\{(f_{y,x})\in \prod_{(x,y)\in
a\times b}\widehat{B}(x,y)| \quad g\cdot f_{y,x}=f_{gy,gx}\mbox{
for all }g\in G,\,x\in a,\,y\in b\right\},$$}} \noindent for all
objects $a,\,b$ of $\widehat{B}/G$. Then we have a canonical
\textit{Galois covering functor} $F:\widehat{B}\rightarrow
\widehat{B}/G$ which assigns to each object $x$ of $\widehat{B}$
its $G$-orbit $Gx$, and, for any objects $x$ of $\widehat{B}$ and
$a$ of $\widehat{B}/G$, $F$ induces natural $K$-linear
isomorphisms
$$\bigoplus_{y\in \widehat{\mathcal{E}}_B,Fy=a}\widehat{B}(x,y)\tilde{\longrightarrow}(\widehat{B}/G)(Fx,a),$$
$$\bigoplus_{y\in \widehat{\mathcal{E}}_B,Fy=a}\widehat{B}(y,x)\tilde{\longrightarrow}(\widehat{B}/G)(a,Fx).$$\\
The finite dimensional algebra $\bigoplus_{a,b\in \widehat{\mathcal{E}}/G}(\widehat{B}/G)(a,b)$ associated to the orbit category $\widehat{B}/G$ is a selfinjective algebra, denoted by $\widehat{B}/G$ and called an {\it orbit algebra} of $\widehat{B}$, with respect to the admissible automorphism group $G$ of $\widehat{B}$. The group $G$ acts also on the category $\mbox{mod}\,\widehat{B}$ of right $\widehat{B}$-modules (identified with contravariant functors from $\widehat{B}$ to $\mbox{mod}\,K$ with finite support) by $gM=M\circ g^{-1}$ for any $M\in \mbox{mod}\,\widehat{B}$ and $g\in G$. Further, we have the \textit{push-down functor} $F_{\lambda}:\mbox{mod}\,\widehat{B}\rightarrow \mbox{mod}\,\widehat{B}/G$ such that $F_{\lambda}(M)(a)=\bigoplus_{x\in a}M(x)$ for a module $M$ in $\mbox{mod}\,\widehat{B}$ and an object $a$ of $\widehat{B}/G$.\\

The following theorem is a consequence of \cite[Lemma 3.5, Theorem 3.6]{G}.
\begin{thm}
Let $B$ be an algebra and $G$ a torsion-free admissible group of
$K$-linear automorphisms of $\widehat{B}$. Then
\begin{enumerate}
\renewcommand{\labelenumi}{(\roman{enumi})}
\item The push-down functor $F_{\lambda}:\mbox{mod}\,\widehat{B}\rightarrow \mbox{mod}\,\widehat{B}/G$ induces an injection from the set of $G$-orbits of isomorphism classes of indecomposable modules in $\mbox{mod}\,\widehat{B}$ into the set of isomorphism classes of indecomposable modules in $\mbox{mod}\,\widehat{B}/G$.
\item The push-down functor $F_{\lambda}:\mbox{mod}\,\widehat{B}\rightarrow \mbox{mod}\,\widehat{B}/G$ preserves the Auslander-Reiten sequences.
\end{enumerate}
\end{thm}

In general, the push-down functor
$F_{\lambda}:\mbox{mod}\,\widehat{B}\rightarrow
\mbox{mod}\,\widehat{B}/G$,
 associated to a Galois covering $F:\widehat{B}\rightarrow\widehat{B}/G$ is not dense (see \cite{AS0}).
  Following \cite{DS1}, a repetitive category $\widehat{B}$ is said to be {\it locally support-finite},
  if for any object $x$ of $\widehat{B}$, the full subcategory of $\widehat{B}$ given by the supports
  $\mbox{supp}\,M$ of all indecomposable modules $M$ in $\mbox{mod}\,\widehat{B}$ with $M(x)\neq 0$, is
   finite. Here, by a {\it support} of a module $M$ in $\mbox{mod}\,\widehat{B}$ we mean the full
   subcategory of $\widehat{B}$ given by all objects $z$ with $M(z)\neq 0$.\\

The following theorem is a consequence of \cite[Proposition 2.5]{DS2} (see also \cite[Theorem]{DS1}).

\begin{thm}
Let $B$ be an algebra with locally support-finite repetitive
category $\widehat{B}$, and $G$ be a torsion-free admissible group
of automorphisms of $\widehat{B}$. Then the push-down functor
$F_{\lambda}:\mbox{mod}\,\widehat{B}\rightarrow
\mbox{mod}\,\widehat{B}/G$ is dense. In particular, $F_{\lambda}$
induces an isomorphism of the orbit translation quiver
$\Gamma_{\widehat{B}}/G$ of the Auslander-Reiten quiver
$\Gamma_{\widehat{B}}$ of $\widehat{B}$, with respect to the
action of $G$, and the Auslander-Reiten quiver
$\Gamma_{\widehat{B}/G}$ of $\widehat{B}/G$.
\end{thm}

Let A be a selfinjective algebra, $I$ an ideal of $A$, $B = A/I$ ,
and $e$ an idempotent of $A$ such that $e + I$ is the identity of
$B$. We may assume that $e = e_1 +...+e_n$, where $\{e_i; 1 \leq i
\leq n\}$ is a complete set of orthogonal primitive idempotents of
$A$ which are not contained in $I$. Then such an idempotent $e$ is
uniquely determined by $I$ up to an inner automorphism of $A$, and
we call it a {\it residual identity} of $B$ \cite{SY3}. Note that
$B \cong eAe/eI e$. For an ideal $I$ of a selfinjective algebra
$A$, we consider its {\it left annihilator} $\ell_A(I)=\{a \in A|
ax=0 \text{ for any   } x \in I \}$ and its {\it right
annihilator} $r_A(I )=\{a \in A| xa=0 \text{   for any   } x\in
I\}$. Following \cite[(2.1)]{SY3} the ideal $I$ is said to be {\it
deforming} if $eI e = \ell_{eAe}(I ) = r_{eAe}(I )$ and $A/I$ is
triangular. The following lemma was proved in \cite[Lemma
5.1]{SY2}.

\begin{lem} Let $A$ be a selfinjective algebra, $e$ an idempotent of $A$, and
assume that $\ell_A(I ) = Ie$ or $r_A(I ) = eI$. Then $e$ is a
residual identity of the quotient algebra $A/I$.
\end{lem}

Moreover, the following result was obtained in
\cite[Proposition 2.3]{SY3}.
\begin{prop} \label{prop 4.4}
Let $A$ be a selfinjective algebra, $I$ an ideal of $A$, $B =
A/I$, $e$ a residual identity of $B$, and assume that $IeI = 0$. Then
the following conditions are equivalent.
 \begin{enumerate}[\rm (i)]
 \item $Ie$ is an injective cogenerator in $\mod B$.
 \item $eI$ is an
injective cogenerator in $\mod B^{op}$.
\item $\ell_A(I ) = Ie$.
\item $r_A(I ) = eI$.
\end{enumerate}
 Moreover, under these equivalent conditions, we have $\soc A \subseteq I$ and $eI e =
\ell_{eAe}(I ) = r_{eAe}(I)$.
\end{prop}

We end this section with the criterion which is fundamental in the proof of the main theorem (see \cite[Section 3 and 4]{SY1} and \cite[Theorem 5.3]{SY2}).

\begin{thm} \label{SY}
Let $A$ be a selfinjective algebra over an algebraically closed field $K$. The following conditions are equivalent.
\begin{enumerate}[\rm (1)]
 \item $A$ is isomorphic to an orbit algebra $\widehat{B}/ (\varphi\nu_{\widehat{B}})$, where $B$ is an algebra over $K$ with acyclic quiver $Q_B$ and $\varphi$ is a positive automorphism of $\widehat{B}$.
 \item There is an ideal $I$ of $A$ and an idempotent $e$ of $A$ such that
   \begin{enumerate}[\rm (i)]
     \item $r_A(I)=eI,$
     \item the quiver $Q_{A/I}$ of $A/I$ is acyclic.
   \end{enumerate}
\end{enumerate}
Moreover, in this case, $B$ is isomorphic to $A/I$.
\end{thm}

%%%%%%%%%%%%%%%%%%%%%%%%%%%%%%%%%%%%%%%%%%%%%%%%%%%%%%%%%%%%%%%%%%%%%%%%%%%%%%%%%%%%%%%%%%%%%%%%%%%%%%%%%%%%%%%%%%%%%%%%%%%%%%%%%%%%%%%%%%%%%%%%%%%%%%%%%%%%%%%%%%%%%%%%%%%%%%%%%%%%%%%%%%%%%%%%%%%%%%%%%%%%%%
\section{Selfinjective algebras of strictly canonical type}
The aim of this section is to introduce some results on
selfinjective algebras of strictly canonical type. In particular,
we give an answer to the question when the canonical family of
quasi-tubes of such an algebra is generalized standard.

The following results were established in \cite[Theorem 5.1]{KS1}.

\begin{thm} \label{5.1KS}
Let $B$ be a branch extension (respectively, branch coextension)
of a canonical algebra $C$.  Then there exist algebras $C_q$,
$B^-_q$, $B^+_q$, $B^{\ast}_q$ and $\overline{B}_q$,
$q\in\mathbb{Z}$, and a decomposition
$$\Gamma_{\widehat{B}}=\vee_{q\in\mathbb{Z}}(\mathcal{X}_q\vee \mathcal{C}_q)$$
of the Auslander-Reiten quiver $\Gamma_{\widehat{B}}$ of
$\widehat{B}$ such that  the following statements hold:
\begin{enumerate}[\rm (a)]
\item for each  $q\in\mathbb{Z}$, $\mathcal{X}_q$ is a family of
components of $\Gamma_{\widehat{B}}$ containing exactly one simple
$\widehat{B}$-module $S_q$;
\item for each  $q\in\mathbb{Z}$, $\mathcal{C}_q$ is a family
 $(\mathcal{C}_q(\lambda))_{\lambda\in\mathbb{P}_1(K)}$ of pairwise orthogonal standard
 quasi-tubes of $\Gamma_{\widehat{B}}$ with $s(\mathcal{C}_q(\lambda))+p(\mathcal{C}_q(\lambda))=r(\mathcal{C}_q(\lambda))-1$,
 for any $\lambda\in\mathbb{P}_1(K)$;
\item for each  pair $p,\,q\in\mathbb{Z}$ with $p<q$,
 we have $\mbox{Hom}\,_{\widehat{B}}(\mathcal{X}_q,\mathcal{X}_p\vee \mathcal{C}_p)=0$
 and $\mbox{Hom}\,_{\widehat{B}}(\mathcal{C}_q,\mathcal{X}_p\vee \mathcal{C}_p\vee\mathcal{X}_{p+1})=0$;
\item for each  $q\in\mathbb{Z}$, $C_q$ is a canonical algebra,
 $B^-_q$ is a maximal branch coextension of $C_q$ in $B^{\ast}_q$,
 $B^+_q$ is a maximal branch extension of $C_q$ in $B^{\ast}_q$,
 and $B^{\ast}_q$ is a quasi-tube enlargement of $C_q$; moreover,
 $C_q$, $B^-_q$, $B^+_q$, $B^{\ast}_q$ and $\overline{B}_q$ are
full convex subcategories of
  $\widehat{B}$ with $\widehat{B^-_q}=\widehat{B}=\widehat{B^+_q}$, $\nu_{\widehat{B}}(C_q)=C_{q+2}$, $\nu_{\widehat{B}}(B^-_q)=B^-_{q+2}$,
  $\nu_{\widehat{B}}(B^+_q)=B^+_{q+2}$, $\nu_{\widehat{B}}(B^{\ast}_q)=B^{\ast}_{q+2}$, $\nu_{\widehat{B}}(\overline{B}_q)=\overline{B}_{q+2}$.
\item for each  $q\in\mathbb{Z}$, $\mathcal{C}_q$ is the canonical $\mathbb{P}_1(K)$-family
 of quasi-tubes of $\Gamma_{B^{\ast}_q}$, obtained from the canonical $\mathbb{P}_1(K)$-family
 $\mathcal{T}^-_q$ of coray tubes of $\Gamma_{B^-_q}$ by infinite rectangle insertions,
 and from the canonical $\mathbb{P}_1(K)$-family $\mathcal{T}^+_q$ of ray tubes
 of $\Gamma_{B^+_q}$ by infinite rectangle coinsertions;
\item for each $q \in \mathbb{Z}$, $\mathcal{X}_q$ consists of
indecomposable $\overline{B}_q$-modules;
\item for each  $q\in\mathbb{Z}$, we have
 $\nu_{\widehat{B}}(\mathcal{X}_q)=\mathcal{X}_{q+2}$ and
 $\nu_{\widehat{B}}(\mathcal{C}_q)=\mathcal{C}_{q+2}$;
\item for each  $q\in\mathbb{Z}$, $\mbox{Hom}\,_{\widehat{B}}(S_q,\mathcal{C}_q(\lambda))\neq 0$
 for all $\lambda\in\mathbb{P}_1(K)$, and $\mbox{Hom}\,_{\widehat{B}}(S_p,\mathcal{C}_q)= 0$
 for $p\neq q$ in $\mathbb{Z}$;
\item for each  $q\in\mathbb{Z}$, $\mbox{Hom}\,_{\widehat{B}}(\mathcal{C}_q(\lambda),S_{q+1})\neq 0$
 for all $\lambda\in\mathbb{P}_1(K)$, and $\mbox{Hom}\,_{\widehat{B}}(\mathcal{C}_q,S_p)= 0$
 for $p\neq q+1$ in $\mathbb{Z}$;
\item for each  $q\in\mathbb{Z}$, we have $\Omega_{\widehat{B}}(\mathcal{C}^s_{q+1})=\mathcal{C}^s_q$
 and $\Omega_{\widehat{B}}(\mathcal{X}^s_{q+1})=\mathcal{X}^s_q$.
\item $\widehat{B}$ is locally support-finite.
\end{enumerate}
\end{thm}

Moreover, there is the following description of torsion-free
admissible groups of automorphisms of $\widehat{B}$
\cite[Proposition 5.2]{KS1}.

\begin{prop} \label{5.2}
Let $B$ be a branch extension (respectively, branch coextension)
of a canonical algebra $C$.  Then there exists a strictly positive
automorphism $\varrho_{\widehat{B}}$ of $\widehat{B}$ such that
the following statements hold:
\begin{enumerate}[\rm (i)]
\item $\varrho_{\widehat{B}}=\nu_{\widehat{B}}$ or $\varrho_{\widehat{B}}^2=\nu_{\widehat{B}}$;
\item every torsion-free admissible group $G$ of  automorphisms of $\widehat{B}$ is an
 infinite cyclic group generated by a strictly positive automorphism $\sigma\varrho_{\widehat{B}}^s$,
 for some integer $s\geq 1$ and some rigid automorphism $\sigma$ of $\widehat{B}$.
\end{enumerate}
\end{prop}
Preserving the above notation, for a canonical algebra $C$, it
follows from Proposition \ref{prop3.?} that $\Gamma_B$ contains a
generalized standard $\mathbb{P}_1(K)$-family $\mathcal{C}$ of
quasi-tubes. We consider $\mathcal{C}=\mathcal{C}_0$ as a family
of components of $\Gamma_{\widehat{B}}$.

Recall that, following \cite{KS1}, a selfinjective algebra $A$ of
the form $\widehat{B}/G$, where $B$ is a branch extension
(equivalently, branch coextension) of a canonical algebra $C$ and
$G$ is an infinite cyclic group generated by a strictly positive
automorphism of $\widehat{B}$, is called a \textit{selfinjective
algebra of strictly canonical type}. The structure and homological
properties of the Auslander-Reiten quivers  of selfinjective
algebras of strictly canonical type were described in
\cite[Theorem 5.3]{KS1}. In particular, for a selfinjective
algebra $A$ of strictly canonical type its Auslander-Reiten quiver
has a decomposition
$$\Gamma_A=\bigvee_{q\in\mathbb{Z}/n\mathbb{Z}}(\mathcal{X}^A_q\vee \mathcal{C}^A_q),$$
for some positive integer $n$, and, for each
$q\in\mathbb{Z}/n\mathbb{Z}$,
$\mathcal{C}^A_q=(\mathcal{C}^A_q(\lambda))_{\lambda\in
\mathbb{P}_1(K)}$ is a $\mathbb{P}_1(K)$-family of quasi-tubes
with $s(\mathcal{C}^A_q(\lambda))+p(\mathcal{C}^A_q(\lambda))=
r(\mathcal{C}^A_q(\lambda))-1$ for each $\lambda\in
\mathbb{P}_1(K)$, and $\mathcal{X}^A_q$ is a family of components
containing exactly one simple module $S_q$. Moreover, we have the
following proposition which is an immediate consequence of
properties of push-down functor $F_{\lambda}: \mod
\widehat{B}\rightarrow \mod A$.

\begin{prop}\label{prop1}
Let $A=\widehat{B}/G$ be a selfinjective algebra of strictly
canonical type. Then $\mathcal{C}^A_0=F_{\lambda}(\mathcal{C}_0)$
is a $\mathbb{P}_1(K)$-family of quasi-tubes in $\Gamma_A$
maximally saturated by simple and projective modules.
\end{prop}

We are now in a position to prove the following equivalence.

\begin{prop} \label{prop2}
Let $A=\widehat{B}/G$, where $B$ is a branch extension
(respectively, coextension) of a canonical algebra $C$ with
respect to the canonical $\mathbb{P}_1(K)$-family of stable tubes
and $G$ an admissible group of automorphisms of $\widehat{B}$
generated by a strictly positive automorphism of $\widehat{B}$.
Then the following statements are equivalent.
\begin{enumerate}[\rm (i)]
\item The canonical $\mathbb{P}_1(K)$-family $\mathcal{C}_0^A
=F_{\lambda}(\mathcal{C}_0)$ of quasi-tubes of $\Gamma_A$ is
generalized standard.
\item $G=(\varphi \nu_{\widehat{B}})$, where
$\varphi$ is strictly positive, or $G=(\varphi\nu_{\widehat{B}})$,
where $\varphi$ is rigid and $B$ is a canonical algebra.
\end{enumerate}
\end{prop}

\begin{proof}
Let $C$ be a canonical algebra and $\mathcal{T}^C$ the canonical
$\mathbb{P}_1(K)$-family of pairwise orthogonal standard stable
tubes of $\Gamma_C$. Since the classes of repetitive algebras of
branch extensions and branch coextensions of $C$ with respect to
the canonical $\mathbb{P}_1(K)$-family of stable tubes coincide
(see \cite[Section 4]{KS1}), invoking Proposition \ref{prop A},
we may assume that $B$ is a branch coextension of $C$.

Let $A$ be an orbit algebra $\widehat{B}/G$, where $G$ is
generated by a strictly positive automorphism $g$ of
$\widehat{B}$. Following Theorem \ref{5.1KS}(a) and (b), the
Auslander-Reiten quiver $\Gamma_{\widehat{B}}$ of $\widehat{B}$
has a decomposition:
\[\Gamma_{\widehat{B}}=\vee_{q\in\mathbb{Z}}(\mathcal{X}_q\vee \mathcal{C}_q),\]
such that, for each $q \in \mathbb{Z}$, $\mathcal{X}_q$ is a
family of components containing exactly one simple
$\widehat{B}$-module $S_q$, and $\mathcal{C}_q$ is a family
$(\mathcal{C}_q(\lambda))_{\lambda\in\mathbb{P}_1(K)}$ of pairwise
orthogonal standard quasi-tubes. We set $C=C_0$ and $B=B_0^-$.
Applying now Theorem \ref{5.1KS}(e),(h) and (i), we know that
there exist, for any $q \in \mathbb{Z}$, an indecomposable
$C_q$-module in $\mathcal{C}_q$ with $S_{q+1}$ in its top and an
indecomposable $C_{q+1}$-module in $\mathcal{C}_{q+1}$ which has
$S_{q+1}$ in the socle. Hence $\Hom_{\widehat{B}}(\mathcal{C}_0,
\mathcal{C}_1) \neq 0$.

Assume $\mathcal{C}^A_0=F_{\lambda}(\mathcal{C}_0)$ is a
generalized standard family of quasi-tubes. Since
$\nu_{\widehat{B}}(\mathcal{C}_q)=\mathcal{C}_{q+2}$ and
$\Hom_{\widehat{B}}(\mathcal{C}_0, \mathcal{C}_1) \neq 0$,
applying Proposition \ref{5.2}, we conclude that $g=\varphi
\nu_{\widehat{B}}$, where $\varphi$ is a positive automorphism of
$\widehat{B}$. Recall that $\mathcal{C}_0$ is the canonical
$\mathbb{P}_1(K)$-family of quasi-tubes obtained from the
canonical $\mathbb{P}_1(K)$-family $\mathcal{T}^B$ of coray tubes
of $\Gamma_B$ by iterated infinite rectangle insertions (see
Theorem \ref{5.1KS}(e)). Observe that $\mathcal{C}_0$ contains no
projective-injective modules if and only if $\mathcal{T}^B$
contains no injective modules, that is, $B=C$ (see \cite[Section 2
and 3]{KS1}). In this situation $\Hom_{\widehat{B}}(\mathcal{C}_0,
\mathcal{C}_2) = 0$, since $\supp \mathcal{C}_0 \cap \supp
\mathcal{C}_2= \supp \mathcal{C}_0 \cap \supp
\nu_{\widehat{B}}(\mathcal{C}_0) = \emptyset$. Suppose now that
$\mathcal{C}_0$ contains a projective-injective module,
equivalently $B \neq C$. Then by Theorem \ref{5.1KS}(g),
$\mathcal{C}_p$, for any even $p \in \mathbb{Z}$, contains a
projective-injective module. Let $P$ be a projective-injective
module which belongs to $\mathcal{C}_2$. Clearly, then $P/\soc P$
belongs to $\mathcal{C}_2$. From Theorem \ref{5.1KS}(j) we obtain
that the simple socle $\soc P$ of $P$ belongs to $\mathcal{C}_1$.
Again, by Theorem \ref{5.1KS}(j), we get that the projective cover
$P'$ of $\soc P$ belongs to $\mathcal{C}_0$, because $\rad
P'=\Omega_{\widehat{B}}(\soc P)$ belongs to $\mathcal{C}_0$.
Hence, there is a non-zero homomorphism $f: P' \rightarrow P$
which implies that $\Hom_{\widehat{B}}(\mathcal{C}_0,
\mathcal{C}_2) \neq 0$ for $B\neq C$. Therefore, if $\varphi$ is
rigid, then $B=C$. Summing up, we conclude that
(i) implies (ii).\\

Assume now that (ii) holds. Suppose that $M,N$ are indecomposable
$\widehat{B}$-modules belonging to $\mathcal{C}_p$, for some $p
\in \mathbb{Z}$. From the description of subcategories $B^-_q,
B^+_q$ of $\widehat{B}$, we know that $\supp \mathcal{C}_q \cap
\supp \mathcal{C}_{q+3}= \emptyset$ for any $q \in \mathbb{Z}$
(see the proof of \cite[Theorem 5.1]{KS1}). Then, by the
assumption imposed on a strictly positive generator $g$ of the
group $G$, we obtain that $\supp ^{g^i}M \cap \supp N = \emptyset$
for any integer $i \neq 0$. Since the push-down functor
$F_{\lambda}:\mod \widehat{B} \rightarrow \mod A$ is dense, there
are the following natural isomorphisms of $K$-vector spaces
%\begin{center}
\[\bigoplus_{i \in \mathbb{Z}}\Hom_{\widehat{B}}(^{g^i}M,N) \cong \Hom_A(F_{\lambda}(M), F_{\lambda}(N)),\]
\[\bigoplus_{i \in \mathbb{Z}}\Hom_{\widehat{B}}(M, {^{g^i}N}) \cong \Hom_A(F_{\lambda}(M), F_{\lambda}(N)),\]
%\end{center}
for any indecomposable modules $M, N$  in $\mod \widehat{B}$.

Let $X,Y$ be modules in
$\mathcal{C}^A_0=F_{\lambda}(\mathcal{C}_0)$. Then
$X=F_{\lambda}(M), Y=F_{\lambda}(N)$ for some $M \in
\mathcal{C}_p$, $N \in\mathcal{C}_q$, and clearly
$F_{\lambda}(\mathcal{C}_p)=F_{\lambda}(\mathcal{C}_q)=\mathcal{C}_0$.
Without loss of generality we may assume that $p=q=0$. Thus
$\Hom_A(X,Y) \cong \bigoplus_{i \in
\mathbb{Z}}\Hom_{\widehat{B}}(M, {^{g^i}N}) \cong
\Hom_{\widehat{B}}(M,N)$. Since $\mathcal{C}_0$ is a family of
pairwise orthogonal standard quasi-tubes, we have
$\rad^{\infty}_{\widehat{B}}(M,N)=0$, and hence
$\rad^{\infty}_A(X,Y)=0$. This shows that (ii) implies (i).
\end{proof}

%%%%%%%%%%%%%%%%%%%%%%%%%%%%%%%%%%%%%%%%%%%%%%%%%%%%%%%%%%%%%%%%%%%%%%%%%%%%%%%%%%%%%%%%%%%%%%%%%%%%%%%%%%%%%%%%%%%%%%%%%%%%%%%%%%%%%%%%%%%%%%%%%%%%%%%%%%%%%%%%%%%%%%%%%%%%%%%%%%%%%%%%%%%%%%%%%%%%%%%%%%%

\section{Proof of Theorem 1.1}

The implication (ii) $\Rightarrow$ (i) of the main theorem is an immediate consequence of Propositions \ref{prop1} and \ref{prop2}.\\

We prove now the implication (i) $\Rightarrow$ (ii). Suppose that
$\mathcal{C}=(\mathcal{C}_{\lambda})_{{\lambda} \in {\Lambda}}$ is
a generalized standard family of quasi-tubes maximally saturated
by simple and projective modules in the Auslander-Reiten quiver
$\Gamma_A$ of  a selfinjective algebra $A$. We recall that the
annihilator of the family of components
$\mathcal{C}=(\mathcal{C}_{\lambda})_{{\lambda} \in {\Lambda}}$ is
the intersection $\ann_A(\mathcal{C})= \bigcap _{X\in \mathcal{C}}
\ann_A (X)$ of the annihilators of all indecomposable $A$-modules
$X$ belonging to $\mathcal{C}$. Consider the quotient algebra
$D=A/ \ann_A(\mathcal{C})$. Then the family $\mathcal{C}$ is a
generalized standard faithful family of quasi-tubes in $\Gamma_D$
maximally saturated by simple and projective modules. We claim
that $D$ is a quasi-tube enlargement of a canonical algebra $C$.
Namely, by definition, a quasi-tube is a connected translation
quiver obtained from a stable tube by an iterated application of
admissible operations $(\ad  1), (\ad  2)$ and their dual
versions. Assume that, for each ${\lambda} \in {\Lambda}$, a
quasi-tube $\mathcal{C}_{\lambda}$ is obtained, as a translation
quiver, from a stable tube $\mathcal{T}_{\lambda}$ by means of the
above operations. This allows us to consider the family
$T(\mathcal{C})=\bigcup_{{\lambda} \in {\Lambda}}
T(\mathcal{C}_{\lambda})$ of indecomposable modules in
$\mathcal{C}$ such that, for each ${\lambda} \in {\Lambda}$,
$T(\mathcal{C}_{\lambda})$ corresponds to all vertices of the
stable tube $\mathcal{T}_{\lambda}$. Let $C=D/
\ann_D(T(\mathcal{C}))$ be a quotient algebra of $D$ by the
annihilator $\ann_D(T(\mathcal{C}))$ of the family
$T(\mathcal{C})$ given as the intersection $\bigcap_{Y \in
T(\mathcal{C})}\ann_D(Y)$ of the annihilators of all modules
belonging to $T(\mathcal{C})$. Then $D$ is a quasi-tube
enlargement of $C$. Note that the modules from $T(\mathcal{C})$
form the family of stable tubes
$\mathcal{T}^C=(\mathcal{T}^C_{\lambda})_{{\lambda} \in
{\Lambda}}$ in $\Gamma_C$, where
$\mathcal{T}^C_{\lambda}=\mathcal{T}_{\lambda}$ for every
${\lambda} \in {\Lambda}$. Clearly,
$(\mathcal{T}^C_{\lambda})_{{\lambda} \in {\Lambda}}$ is a
faithful  generalized standard family of stable tubes in
$\Gamma_C$ (maximally saturated by simple modules). Thus, invoking
Theorem \ref{thm1}, we conclude  that $C$ is a canonical algebra
and then $D$ is a quasi-tube enlargement  of the canonical algebra
$C$. In particular, $\Lambda=\mathbb{P}_1(K)$ and $\mathcal{T}^C$
is the separating canonical $\mathbb{P}_1(K)$-family
$(\mathcal{T}^C_{\lambda})_{\lambda \in \mathbb{P}_1(K)}=
(\mathcal{T}^C_{\lambda})_{{\lambda} \in {\Lambda}}$ of stable
tubes in $\Gamma_C$. Hence, applying  \cite[(3.5)]{AST}, we infer
that there exists a unique maximal branch extension $B$ of $C$
inside $D$, which is obtained from $C$ by an iterated application
of  algebra operations of type $(\ad  1)$ (see also \cite[Theorem
C]{MS2}). Then the Auslander-Reiten quiver $\Gamma_B$ of $B$
contains a faithful $\mathbb{P}_1(K)$-family
$\mathcal{T}^B=(\mathcal{T}^B_{\lambda})_{\lambda \in
\mathbb{P}_1(K)}$ of pairwise orthogonal generalized standard ray
tubes (obtained from $\mathcal{T}^C
=(\mathcal{T}^C_{\lambda})_{\lambda \in \mathbb{P}_1(K)}$ by an
iterated application of translation quiver operations of type
$(\ad 1)$). Moreover, $D$ is obtained from $B$ by an iterated
application of admissible algebra operations of types $(\ad
1^{\ast})$, $(\ad  2^{\ast})$, and $\mathcal{C}$ from
$\mathcal{T}^B$ by an iterated application of translation quiver
operations of types $(\ad 1^{\ast})$, $(\ad 2^{\ast})$. Then
$B=D/\ann_D(\mathcal{T}^B)$, where $\mathcal{T}^B$ is taken as a
family of modules. Thus we conclude that $B=A/\ann_A
(\mathcal{T}^B)$ since
$\ann_D(\mathcal{T}^B)=\ann_A(\mathcal{T}^B) \cap D$.

Let $I=\ann_A(\mathcal{T}^B)$. Then $B=A/I$. We will show now that $I$ satisfies  the conditions (2) of Theorem \ref{SY}. Observe that $Q_{A/I}=Q_B$ is acyclic, because $B$ is a $\mathcal{T}^C$-branch extension of the canonical algebra $C$.

By $J$ we shall denote the trace ideal of the family
$\mathcal{T}^B$ in $A$, that is
 $J=\sum _{h} \Im h$, where $h \in \Hom_{A}(Y, A_A)$ for any  $Y \in
 \mathcal{T}^B$. Since $A_A$ is of finite dimension over $K$, we obtain that
 $J$  is a finite sum $J=\sum _{i=1}^s \Im h_i$ for some
 homomorphisms $h_i \in \Hom_A(Y_i, A_A)$ with $Y_i \in
 \mathcal{T}^B$.
Similarly, by $J'$ we denote the trace ideal of the dual family
$\D (\mathcal{T}^B)$ of left $A$-modules in $A$.

We may choose a complete set of pairwise orthogonal primitive
idempotents $e_1,...,e_r$ of $A$ such that $1_A=e_1+...+e_r$ and
$e=e_1+...+e_n$, for some $n\leqslant r$,  is a residual identity
of $B=A/ I$. Observe that $B\cong eAe/eIe$. We will show that $I$
is a deforming ideal of $A$ with $\ell_A(I)=Ie$ and $r_A(I)=eI$.
We will apply the strategy similar to the proof of \cite[Theorem
7.14]{KaSY}.

\begin{prop}
We have $J \cup J' \subseteq I$.
\end{prop}
\begin{proof}
We know that $\mathcal{T}^B=(\mathcal{T}^B_{\lambda})_{\lambda \in
\mathbb{P}_1(K)}$ is a generalized standard family of ray tubes in
$\Gamma_B$ and the generalized standard family
$\mathcal{C}=(\mathcal{C}_{\lambda})_{\lambda \in
\mathbb{P}_1(K)}$ of quasi-tubes in $\Gamma_A$ is obtained from
$\mathcal{T}^B$ by an iterated application of admissible
translation quiver operations of types $(\ad 1^{\ast})$ and $(\ad
2^{\ast})$ corresponding to the admissible algebra operations of
types $(\ad 1^{\ast})$ and $(\ad 2^{\ast})$ leading from $B$ to
$D$. Then, applying arguments as in the proof of
 \cite[Proposition 7.1]{KaSY}, we prove the required inclusion $J \cup J' \subseteq I$.
\end{proof}

Applying arguments as in the proof of \cite[Lemma 7.2]{KaSY} we obtain the following facts.
\begin{prop} \label{prop 6.2}
We have $\ell_A(I)=J$, $r_A(I)=J'$ and $I=r_A(J)=\ell_A(J')$.
\end{prop}
The following proposition is the key ingredient for proving that
$I$ is a deforming ideal of $A$ such that $\ell_A(I)=Ie$ and
$r_A(I)=eI$.
\begin{prop} \label{prop 6.3}
We have $eIe=eJe=eJ'e$. In particular, $(eIe)^2=0$.
\end{prop}
\begin{proof}
Observe that $J$ is a right $B$-module since $1_{A}-e\in I$
implies that $J(1-e)\subseteq JI=0$ and so $J=Je+J(1_{A}-e)=Je$.
Hence $eJ$ is an ideal of $eAe$ with $eJ\subseteq eIe$, by
Proposition 6.1. We denote by $B'$ the algebra $B'=eAe\slash eJ$.
Note that $e$ is a residual identity of $B'$.

Consider the canonical restriction functor $res_{e}:\mod
A\rightarrow\mod eAe$. Applying $res_{e}$ to the (generalized) standard family
$\mathcal{C}=(\mathcal{C}_{\lambda})_{\lambda\in\Lambda}$ of
quasi-tubes maximally saturated by simple and projective modules
in $\Gamma_{A}$ we obtain the family $\mathcal{T}^{B}$ of ray
tubes in $\Gamma_{eAe}$. Moreover, $\mathcal{T}^{B}$ is sincere
generalized standard in $\Gamma_{eAe}$. Further, $\mathcal{T}^{B}$
is also sincere generalized standard family for the quotient
algebra $B'=eAe\slash eJ$ since $eJ\subseteq eIe$. We shall show
that, in fact, the algebras $B'$ and $B$ are equal.

We shall compare the bound quivers $(Q_{B},I_{B})$,
$(Q_{B'},I_{B'})$ of algebras $B$, $B'$, respectively.  Since
$Q_{B}$ is a subquiver of $Q_{B'}$ with the same set of vertices,
suppose there exists an arrow $\eta: x\rightarrow y$ in $Q_{B'}$
which does not belong to $Q_{B}$. Recall that $Q_{B}$ is of the
form

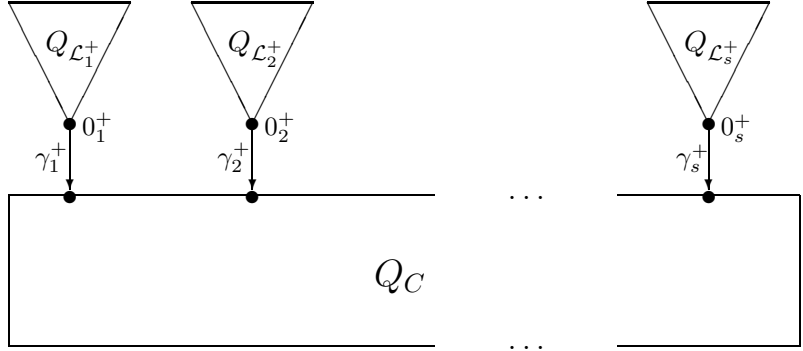
\unitlength0.8cm
\begin{center}

\begin{picture}(13,5.8)

\multiput(0.43,3)(3,0){1}{\footnotesize{$\gamma_1^+$}}
\multiput(3.43,3)(3,0){1}{\footnotesize{$\gamma_2^+$}}
\multiput(10.97,3)(3,0){1}{\footnotesize{$\gamma_s^+$}}

\multiput(0,0)(0,2.5){2}{\line(1,0){7}}
\multiput(10,0)(0,2.5){2}{\line(1,0){3}}
\multiput(8.2,-0.05)(0,2.5){2}{$\ldots$}
\multiput(1,3.6)(3,0){2}{\vector(0,-1){1}}
\put(11.5,3.6){\vector(0,-1){1}}
\multiput(0.87,2.35)(3,0){2}{$\bullet$}
\put(11.37,2.35){$\bullet$}
\multiput(0,0)(13,0){2}{\line(0,1){2.5}}

\multiput(1,3.7)(3,0){2}{\line(1,2){1}}
\multiput(0,5.7)(3,0){2}{\line(1,-2){1}}
\multiput(0,5.7)(3,0){2}{\line(1,0){2}}

\put(11.5,3.7){\line(1,2){1}}
\put(10.5,5.7){\line(1,-2){1}}
\put(10.5,5.7){\line(1,0){2}}

\multiput(0.87,3.55)(3,0){2}{$\bullet$}
\put(11.37,3.55){$\bullet$}

\put(0.6,4.9){\small{$Q_{\mathcal{L}^+_1}$}}
\put(3.6,4.9){\small{$Q_{\mathcal{L}^+_2}$}}
\put(11.1,4.9){\small{$Q_{\mathcal{L}^+_s}$}}
\put(1.2,3.5){\footnotesize{$0^+_1$}}
\put(4.2,3.5){\footnotesize{$0^+_2$}}
\put(11.7,3.5){\footnotesize{$0^+_s$}}
\put(6,1){\large{$Q_C$}}

\end{picture}
\end{center}
where $Q_{\mathcal{L}^+_1}$,
$Q_{\mathcal{L}^+_2}$,...,$Q_{\mathcal{L}^+_s}$ are the quivers of
the branches $\mathcal{L}^+_1$, $\mathcal{L}^+_2$,...,
$\mathcal{L}^+_s$, respectively, with the vertex $0_{i}^{+}$ such
that $\rad P(0_i^+)$, for $i \in \{1,...,s\}$, are pairwise
nonisomorphic modules lying on the mouths of stable tubes from
canonical family $\mathcal{T}^C$  in $\Gamma_C$ (see \cite[Chapter
XV.3]{SS2}). For $Q_C$ we shall use the notation from Section 2.
By $B''$ we denote a quotient algebra of $B'$ such that the set of
arrows of $Q_{B''}$ consists of all arrows of $Q_B$ and
additionally the arrow $\eta$. Hence we have a sequence of algebra
epimorphisms $B'\rightarrow B''\rightarrow B$. This implies that
$\mathcal{T}^{B}$ is a sincere generalized standard family of ray
tubes in $\Gamma_{B''}$. We have the following cases to consider.\\

\noindent (1) Assume that $x\in Q_{\mathcal{L}_{i}}$ for some
$i\in\{1,...,s\}$. Since $\mathcal{T}^{B}$ contains all projective
$B$-modules $P(a)$ for $a \in Q_{\mathcal{L}_{i}}$ and
$\mathcal{T}^B$ is a family of ray tubes in $\Gamma_{B'}$, a
projective $B$-module $P(x)$ is also a projective $B'$-module.
Therefore, if $\eta: x\rightarrow y$ in $Q_{B'}$ then $\eta$
belongs to $Q_B$, a contradiction with an assumption imposed on
$\eta$. If now $y\in Q_{\mathcal{L}_{i}}$, for some
$i\in\{1,...,s\}$, then there exists a homomorphism
$f:P(y)\rightarrow P(x)$ in $\mod B'$ given by the formula
$f(-)=\eta\cdot -$.
Hence $Im f=\eta e_{y}B'=\eta eAe$, because $\mathcal{T}^B$ is a family of ray tubes in $\Gamma_{eAe}$.
Therefore $\eta eAe\subseteq eJ$, contradiction with the assumption that $\eta\in Q_{B'}$. \\
\noindent (2) Let now $x \in Q_C$ and $y \in
Q_C\backslash\{\omega\}$.  Assume $y=(i,k)$ for some
$i\in\{1,...,m\}$ and $k\in\{1,...,p_{i}-1\}$. Then $S(i,k)$ has
in $\mod B''$ a minimal injective presentation of the form
$$0\rightarrow S(i,k)\rightarrow I(i,k)\rightarrow I(i,k+1)\oplus I(x)\oplus I(0^+),$$
where $I(0^+)=0$ if there is no branch extension of $C$ at
$S(i,k)$. Then using the quasi-inverse $\nu^{-1}_{B''}$ of the
Nakayama functor $\nu_{B''}$, we obtain the following exact
sequence
$$0\rightarrow P(i,k)\rightarrow
P(i,k+1)\oplus P(x)\oplus P(0^+)
\rightarrow\tau^{-}_{B''}S(i,k)\rightarrow 0,$$ where we assume
that $P(0^+)=0$ if $I(0^+)=0$. Then the socle of
$\tau^{-}_{B''}S(i,k)$ contains an additional direct summand
$S(x)$. Therefore, $\tau^{-}_{B''}S(i,k) \neq \tau^{-}_{B}S(i,k)$,
and hence $\mathcal{T}_B$ is not a family of ray tubes in
$\Gamma_{B''}$ and neither  is in $\Gamma_{B'}$, a contradiction.
In the case of $y=0$, we repeat the above arguments for a
nonsimple module $F_i$ from the mouth of
a stable tube from $\mathcal{T}^C$.\\
\noindent (3) Assume that $x=(i,k)$, for some $i\in\{1,...,m\}$
and $k\in\{1,...,p_{i}-1\}$, and $y=\omega$. Then we show
analogously that $\tau^-_{B''}(\tau_BS(x))\neq S(x)$ since
$B''$-module $\tau^-_{B''}(\tau_BS(x))$ contains $\eta B$ as a
submodule. Again, we get a contradiction.\\
\noindent (4) Let $\eta: \omega \rightarrow \omega$. By (1), (2)
and (3) we conclude that
$\eta,\alpha_{1,p_{1}},...,\alpha_{m,p_{m}}$ are all arrows that
start at $\omega$ in $Q_{eAe}$. Denote by $\varrho_i$ the path
$\alpha_{i,p_i}...\alpha_{i,2}\alpha_{i,1}$, for any
$i\in\{1,...,m\}$. Observe that $\eta\varrho_i$ belongs to the
$K$-vector space $e_{\omega}Ae_{0}$ generated by $\varrho_1$ and
$\varrho_2$, otherwise using the canonical restriction functor
$\res_{e'}:\mod eAe\rightarrow\mod e'Ae'$ for
$e'=e_{\omega}+e_{0}$, we obtain that $e'Ae'$ is a wild algebra
(see arguments from the proof of Theorem 2.2 and
\cite[(XVIII.1.6)]{SS2}). Assume $\eta\varrho_i \neq 0$ in $B''$,
for some $i \in \{1,\ldots ,m\}$. Then $\eta\varrho_i
=a_1\varrho_1+a_2\varrho_2$ for some $a_1,a_2 \in K$ such that
$a_1^2+a_2^2>0$. Consider the nonsimple module $E^{(0)}$ from the
mouth of the stable tube $\mathcal{T}^C_0$ from the family
$\mathcal{T}^C$ (see Section 2). Note that $E^{(0)}_{\varrho_2}=0$
and hence $E^{(0)}_{a_1\varrho_1}=E^{(0)}_{\eta\varrho_i} =0$. But
$E^{(0)}_{a_1\varrho_1}(e_{\omega})=a_1\varrho_1 \neq 0$ and we
conclude that $\eta\varrho_i=0$ for any $i \in \{1, \ldots, m\}$.
Let now $k_i\in \{0,1,\ldots, p_i\}$ be the minimal integer such
that $\eta\alpha_{i, p_i}\ldots \alpha_{i,k_i} \neq 0$ in $B''$
where we put $\alpha_{i,p_i}\ldots\alpha_{i,0}=e_{\omega}$. Then,
for any $i\in \{1, \ldots, m \}$, there exists a nonzero
homomorphism $f_i:S(i,k_i-1)[p_i-k_i+1] \rightarrow P(\omega)$ in
$\mod eAe$, where $S(i,k_i-1)[p_i-k_i+1]$ is a module of
$\mathcal{T}^C$-length $p_i-k_i+1$ lying on a ray starting at
mouth module $S(i, k_i-1)$ in $\mathcal{T}^C$. Note that $Im
f_i=\eta eAe$. Thus $\eta eAe \subset eJ$ and $\eta \notin
Q_{B'}$.\\
\noindent (5) Let now $\eta_1, \eta_2, \ldots, \eta_r$, for some
$r \geq 1$, be all arrows in $Q_{B'}$ which start at $0$ and end
in $\omega$.

Consider the Galois covering $F: \widetilde{B'}\rightarrow B'$
with an infinite cyclic group $\mathbb{Z}$. Then $\widetilde{B'}$
is a locally bounded $K$-category and it follows from
\cite[Section 2]{BG} that $\widetilde{B'}\cong
KQ_{\widetilde{B'}}/ I_{\widetilde{B'}}$, where
$Q_{\widetilde{B'}}$ is a connected, locally finite, acyclic
quiver with $I_{\widetilde{B'}}$ an admissible ideal of the path
category $KQ_{\widetilde{B'}}$ of $Q_{\widetilde{B'}}$. Thus
quiver $\widetilde{Q_{B'}}$ is of the form \unitlength1cm
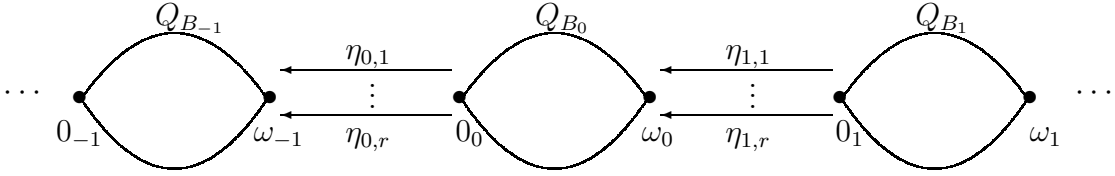
\begin{center}
\begin{picture}(13,3)

\put(-1,0.7){$\cdots$} \put(13.1,0.7){$\cdots$}
%\multiput(0,0)(2.5,0){6}{\line(0,1){1.5}}
%\multiput(0,1.5)(5,0){3}{\line(1,0){2.5}}
%\multiput(0,0)(5,0){3}{\line(1,0){2.5}}
\put(1,1.7){$Q_{B_{-1}}$}
\put(6,1.7){$Q_{B_{0}}$}
\put(11,1.7){$Q_{B_{1}}$}
\multiput(4.9,1.1)(5,0){2}{\vector(-1,0){2.25}}
\multiput(4.9,0.5)(5,0){2}{\vector(-1,0){2.25}}
\put(3.8,0.6){$\vdots$} \put(8.8,0.6){$\vdots$}

\put(3.5,1.25){$\scriptsize{\eta_{0,1}}$}
\put(3.5,0.15){$\scriptsize{\eta_{0,r}}$}

\put(8.5,1.25){$\scriptsize{\eta_{1,1}}$}
\put(8.5,0.15){$\scriptsize{\eta_{1,r}}$}

\multiput(-0.1,0.63)(2.5,0){6}{$\bullet$}
\put(2.3,0.16){$\scriptsize{\omega_{-1}}$}
\put(7.4,0.16){$\scriptsize{\omega_{0}}$}
\put(12.5,0.16){$\scriptsize{\omega_{1}}$}
\put(-0.3,0.16){$\scriptsize{0_{-1}}$}
\put(4.95,0.16){$\scriptsize{0_0}$}
\put(9.95,0.16){$\scriptsize{0_1}$}

\qbezier(0,0.68)(1.25,2.5)(2.5,0.68)
\qbezier(0.05,0.68)(1.25,-1.1)(2.45,0.68)

\qbezier(5,0.68)(6.25,2.5)(7.5,0.68)
\qbezier(5.05,0.68)(6.25,-1.1)(7.45,0.68)

\qbezier(10,0.68)(11.25,2.5)(12.5,0.68)
\qbezier(10.05,0.68)(11.25,-1.1)(12.45,0.68)

\end{picture}
\end{center}
where, for all $k \in \mathbb{Z}$ and $1 \leq i \leq n$, we have
$Q_{B_k}=Q_B$, $\eta_{k,i}:0_k \rightarrow \omega_{k-1}$,  and the
generators of $I_{B_k}$ belong to the set generators of
$I_{\widetilde{B'}}$. By $B'_-$ we denote the full subcategory of
$\widetilde{B'}$ whose objects are the objects of $B_k$ for all
integer $k\leq 0 $. Then the Auslander-Reiten quiver $B'_-$ has a
form
$$\Gamma_{B'_-}= \mathcal{P}^{B'_-} \cup \mathcal{T}^{B_0} \cup \mathcal{Q}^{B_0},$$
where $\mathcal{P}^{B'_-}$ is a family of components containing all indecomposable modules $X$ such that $\res_e(X)$, for $e$ being the residual identity of $B_0$,
is zero or belongs to the family $\mathcal{P}^{B_0}$ of $\Gamma_{B_0}$, $\mathcal{T}^{B_0}=\mathcal{T}^{B}$ and $\mathcal{Q}^{B_0}=\mathcal{Q}^{B}$.
By $P$  we denote the projective $\widetilde{B'}$-module $P_{\widetilde{B'}}(0_1)$  at the vertex $0_1$ (vertex $0$ in $Q_{B_1}$).
Let now $R=B'_-[\rad P]$ be a one-point extension of $B'_-$ by the radical of $P$.
Since $\mathcal{T}^{B_0}$ remains a family of ray tubes in $\Gamma_R$ (and in $\Gamma_{\widetilde{B'}}$),
we conclude that $\rad P \in \add \mathcal{Q}^{B_0}$ (see \cite[Theorem XV.1.6]{SS2}). Consider the projective cover $f:P' \rightarrow \rad P$ of $\rad P$ in $\mod B_0$. Then $f$ factorizes through the additive subcategory $\add\mathcal{T}^{B_0}$ of $\mod B_0$, because $\mathcal{T}^{B_0}$  is a separating family of ray tubes in $\Gamma_{B_0}$. Therefore, there is a module $M \in \add \mathcal{T}^{B_0}$ and an epimorphism $h: M \rightarrow \rad P$ in $\mod \widetilde{B'}$. Further, there exists an epimorphism $h': M \rightarrow \rad P$ in $\mod B'$ because for the push-down functor $F_{\lambda}: \mod \widetilde{B'} \rightarrow \mod B' $ we have $F_{\lambda}(M)=M$ and $F_{\lambda}(\rad P)=\rad P$. Observe that by (1)-(4), $P$ is a projective $eAe$-module. Hence $\Im h'=\sum_{i=1}^r \eta_ieAe$ and $\eta_1, \ldots, \eta_r \in eJ$, a contradiction.\\

To sum up, we obtain that $eIe=eJ$. Hence
$\mathcal{T}^B$ is faithful generalized standard family of ray
tubes in $\mod B'$ because $eIe/eJ=\ann_{B'}(\mathcal{T}^B)=0$. We
show analogously that $eIe=J'e$. Applying now Proposition
\ref{prop 6.2}, we have $(eIe)^2=eJeeIe=eJeIe=(eJe)Ie=eJIe=0$.

\end{proof}
 We note that although Proposition \ref{prop 6.3} is
the analogue of Lemma 7.3 in \cite{KaSY} their proofs are
different because the family $\mathcal{C}$ of quasi-tubes is
assumed only to be generalized standard whereas in \cite{KaSY} the
quasi-tubes in $\mathcal{C}$ consist of modules which do not lie
on infinite short cycles. But having Proposition 6.3, we may
proceed as in \cite[Section 7]{KaSY}, and, using Lemmas 7.4-7.12
of \cite{KaSY}, prove the following analogue of \cite[Proposition
7.13]{KaSY}.

\begin{prop} \label{prop 6.4}
We have $Ie=J$, $eI=J'$, and $eIe=J \cap J'$.
\end{prop}
This allows us to prove the desired proposition.
\begin{prop}
$I$ is a deforming ideal of $A$ with $\ell_A(I)=Ie$ and $r_A(I)=eI$.
\end{prop}
\begin{proof}
From Proposition \ref{prop 6.2} and  \ref{prop 6.4} we know that
$\ell_A(I)=J=Ie$ and $r_A(I)=J'=eI$. In particular, we have
$IeI=0$. Therefore, applying Proposition \ref{prop 4.4}, we get
$eIe=\ell_{eAe}(I)= r_{eAe}(I)$. Since $Q_{A/I}=Q_B$ is acyclic,
this shows that $I$ is a deforming ideal of $A$.
\end{proof}

We complete now the proof of implication  (i) $\Rightarrow$ (ii)
of Theorem 1.1. Since the ideal $I$ and the idempotent $e$ satisfy
condition (2) in Theorem 4.5, we conclude that $A$ is isomorphic
to an orbit algebra $\widehat{B}/ (\varphi\nu_{\widehat{B}})$,
where $\varphi$ is a positive automorphism of $\widehat{B}$.
Finally, applying Proposition \ref{prop2}  we infer that either
$\varphi$ is strictly positive or $\varphi$ is rigid and $B$ is a
canonical algebra, as required in (ii).

\section{Examples}

The following examples illustrate the statements of Theorem 1.1
and Corollary 1.2.\\

\noindent{\scshape{Example 7.1.}}
Let $B=KQ_B/I_B$, where $Q_B$ is the quiver

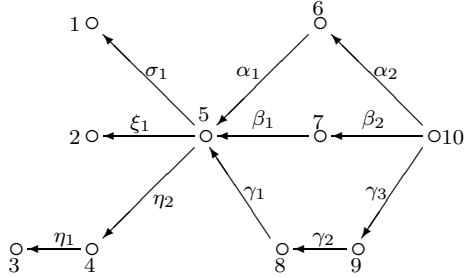
\unitlength1cm
\begin{center}
\begin{picture}(7,4)
\multiput(1,0.5)(1,-1){1}{$\circ$}
\multiput(2,0.5)(0,1.5){3}{$\circ$}
\multiput(3.5,2)(1.5,0){3}{$\circ$}
\multiput(4.5,0.5)(1,1){1}{$\circ$}
\multiput(5,3.5)(1,1){1}{$\circ$}
\multiput(5.5,0.5)(1,1){1}{$\circ$}

\multiput(3.45,2.1)(1.5,0){3}{\vector(-1,0){1.2}}
\multiput(1.95,0.6)(1.5,0){1}{\vector(-1,0){0.7}}
\multiput(5.45,0.6)(1.5,0){1}{\vector(-1,0){0.7}}

\multiput(3.45,2.25)(3,0){2}{\vector(-1,1){1.2}}

\multiput(3.45,1.95)(1.5,1.5){2}{\vector(-1,-1){1.2}}

\multiput(4.45,0.75)(1.5,1.5){1}{\vector(-2,3){0.8}}
\multiput(6.45,1.95)(1.5,1.5){1}{\vector(-2,-3){0.8}}

\multiput(1.8,3.5)(0,2){1}{\scriptsize$1$}
\multiput(1.8,2)(0,2){1}{\scriptsize$2$}
\multiput(1,0.3)(0,2){1}{\scriptsize$3$}
\multiput(2,0.3)(0,2){1}{\scriptsize$4$}
\multiput(3.5,2.3)(0,2){1}{\scriptsize$5$}
\multiput(5,3.7)(0,2){1}{\scriptsize$6$}
\multiput(5,2.2)(0,2){1}{\scriptsize$7$}
\multiput(4.5,0.3)(0,2){1}{\scriptsize$8$}
\multiput(5.5,0.3)(0,2){1}{\scriptsize$9$}
\multiput(6.7,2)(0,2){1}{\scriptsize$10$}

\put(4,2.9){\scriptsize$\alpha_1$}
\put(5.8,2.9){\scriptsize$\alpha_2$}
\put(4.2,2.25){\scriptsize$\beta_1$}
\put(5.65,2.25){\scriptsize$\beta_2$}
\put(4.1,1.3){\scriptsize$\gamma_1$}
\put(5,0.7){\scriptsize$\gamma_2$}
\put(5.7,1.3){\scriptsize$\gamma_3$}
\put(2.8,2.9){\scriptsize$\sigma_1$}
\put(1.6,0.7){\scriptsize$\eta_1$}
\put(2.9,1.2){\scriptsize$\eta_2$}
\put(2.6,2.2){\scriptsize$\xi_1$}
\end{picture}
\end{center}
\noindent and $I_B$ is the ideal of the path algebra $KQ_B$ of $Q_B$ generated
 by the elements $\alpha_1\sigma_1$, $\beta_1\xi_1$, $\gamma_1\eta_2$,  $\gamma_3\gamma_2\gamma_1+\alpha_2\alpha_1+\beta_2\beta_1$.
 Denote by $C$ the bound quiver algebra $C=KQ_C/I_C$, where $Q_C$ is the full subquiver of $Q$ given by the vertices
 $5$, $6$, $7$, $8$, $9$, $10$, and $I_C$ is the ideal in the path algebra $KQ_C$ of $Q_C$ generated by
 $\gamma_3\gamma_2\gamma_1+\alpha_2\alpha_1+\beta_2\beta_1$. Then $C$ is the canonical algebra $C(\hbox{\boldmath$p$}, \hbox{\boldmath$\lambda$})$
 with the weight sequence $\hbox{\boldmath$p$}=(2,2,3)$ and the parameter sequence $\hbox{\boldmath$\lambda$}=(\infty,0,1)$.
 Further, $B$ is the branch coextension of $C$ in the sense of \cite[XV.3]{SS2}. Namely, $B=[E_1,\mathcal{L}_1,E_2,\mathcal{L}_2,E_3,\mathcal{L}_3]C$ with $E_1=E^{(\infty)}\in\mathcal{T}^C_{\infty}$, $E_2=E^{(0)}\in\mathcal{T}^C_{0}$, $E_3=E^{(1)}\in\mathcal{T}^C_{1}$,  $\mathcal{L}_1$ the branch given by the vertex $1$, $\mathcal{L}_2$ the branch given by the vertex $2$, $\mathcal{L}_3$ the branch given by the vertices $3,\,4$ and the arrow $\eta_1$. Consider the repetitive algebra $\widehat{B}$ of $B$. Then there exists a  strictly positive automorphism $\varphi_{\widehat{B}}$ of $\widehat{B}$, with $\varphi_{\widehat{B}}^2=\nu_{\widehat{B}}$ such that, for any $k \in \mathbb{Z}$, $\varphi_{\hat{B}}(e_{k,l})=e_{k,l+5}$, if $1 \leq l\leq 5$ and
$\varphi_{\hat{B}}(e_{k,l})=e_{k+1,l-5}=\nu_{\widehat{B}}(e_{k,l-5})$,
if $6 \leq l \leq 10$. Denote by $A$ the orbit algebra
$\widehat{B}/(\varphi_{\widehat{B}})$. Then  $A$ is the bound
quiver algebra $A=KQ/I$, where $Q$ is the quiver \unitlength1.1cm
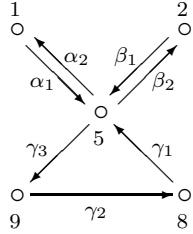
\begin{center}
\begin{picture}(3,3)
\multiput(0.5,0.5)(1,1){3}{$\circ$}
\multiput(2.5,0.5)(2,0){1}{$\circ$}
\multiput(0.5,2.5)(0,1.5){1}{$\circ$}

\multiput(0.75,0.6)(1.5,0){1}{\vector(1,0){1.7}}

\multiput(1.45,1.45)(1,1){1}{\vector(-1,-1){0.7}}
\multiput(2.4,2.5)(1,1){1}{\vector(-1,-1){0.7}}
\multiput(1.8,1.7)(1,1){1}{\vector(1,1){0.7}}

\multiput(2.45,0.75)(1,1){1}{\vector(-1,1){0.7}}
\multiput(1.5,1.8)(1,1){1}{\vector(-1,1){0.7}}
\multiput(0.7,2.4)(1,1){1}{\vector(1,-1){0.7}}

\multiput(0.5,2.75)(0,2){1}{\scriptsize$1$}
\multiput(2.5,2.75)(0,2){1}{\scriptsize$2$}
\multiput(1.5,1.2)(0,2){1}{\scriptsize$5$}
\multiput(2.5,0.2)(0,2){1}{\scriptsize$8$}
\multiput(0.5,0.2)(0,2){1}{\scriptsize$9$}

\put(0.74,1.9){\scriptsize$\alpha_1$}
\put(1.15,2.2){\scriptsize$\alpha_2$}
\put(1.75,2.2){\scriptsize$\beta_1$}
\put(2.2,1.9){\scriptsize$\beta_2$}
\put(1.4,0.35){\scriptsize$\gamma_2$}
\put(2.2,1.1){\scriptsize$\gamma_1$}
\put(0.7,1.1){\scriptsize$\gamma_3$}

\end{picture}
\end{center}

\noindent and $I$ is the ideal of $KQ$ generated by the elements $\gamma_3\gamma_2\gamma_1+\alpha_2\alpha_1+\beta_2\beta_1, \alpha_1\alpha_2, \beta_1\beta_2,\gamma_1\gamma_3,$
$\alpha_2\alpha_1\beta_2\beta_1-\beta_2\beta_1\alpha_2\alpha_1$. Note that $A$ is a symmetric algebra but not a trivial extension of $C$. Hence from Corollary \ref{1.2}, $\Gamma_A$ does not admit a generalized standard family of quasi-tubes maximally saturated by simple and projective modules. Indeed, by \cite[Theorem 5.3]{KS1} $\Gamma_A$ has a decomposition
$$\Gamma_A=\mathcal{X}^A\vee \mathcal{C}^A $$
where
$\mathcal{C}^A=(\mathcal{C}^A(\lambda))_{\lambda\in\mathbb{P}_1(K)}$
is the unique $\mathbb{P}_1(K)$-family of quasi-tubes of
$\Gamma_A$ containing all simple modules and indecomposable
projective modules, except the simple module $S(5)$ and the
projective module $P(5)$ at the vertex 5, which belong to
$\mathcal{X}^A$. For $\lambda=1$ form the parameter sequence
consider the module $E^{(1)}$ from the mouth of a stable tube
$\mathcal{T}^C_1$ in $\Gamma_C$. Then $E^{(1)}$ belongs to
$\mathcal{C}^A(1)$ and there is a nonzero homomorphism $f:
E^{(1)}\rightarrow E^{(1)}$ which factors through $\soc E^{(1)}
\cong S(5) \cong \top E^{(1)}$.
Thus $f \in \rad^{\infty}(\mathcal{C}^{A})$.\\

\noindent{\scshape{Example 7.2.}} Consider the selfinjective
algebra $A=\widehat{B}/(\varphi_{\widehat{B}}^3)$, where $B$ and
$\varphi_{\widehat{B}}$ are as above. Then $A=KQ/I$ is the bound
quiver algebra, where $Q$ is the quiver \unitlength1cm
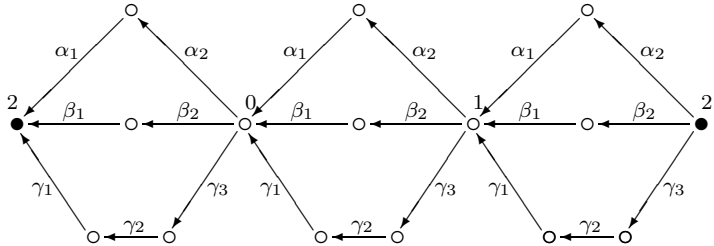
\begin{center}
\begin{picture}(3,4)
\multiput(6.5,2)(-9,0){2}{$\bullet$}
\multiput(5,2)(-1.5,0){5}{$\circ$}
\multiput(5,3.5)(-3,0){3}{$\circ$}
\multiput(5.5,0.5)(-1,0){2}{$\circ$}
\multiput(5.5,0.5)(-1,0){2}{$\circ$}
\multiput(2.5,0.5)(-1,0){2}{$\circ$}
\multiput(-0.5,0.5)(-1,0){2}{$\circ$}

\put(6.6,2.3){\scriptsize{$2$}}
\put(3.6,2.3){\scriptsize{$1$}}
\put(0.6,2.3){\scriptsize{$0$}}
\put(-2.53,2.3){\scriptsize{$2$}}

\multiput(5.8,3)(-3,0){3}{\scriptsize{$\alpha_2$}}
\multiput(4.1,3)(-3,0){3}{\scriptsize{$\alpha_1$}}
\multiput(5.7,2.2)(-3,0){3}{\scriptsize{$\beta_2$}}
\multiput(4.2,2.2)(-3,0){3}{\scriptsize{$\beta_1$}}
\multiput(6.1,1.2)(-3,0){3}{\scriptsize{$\gamma_3$}}
\multiput(5,0.7)(-3,0){3}{\scriptsize{$\gamma_2$}}
\multiput(3.8,1.2)(-3,0){3}{\scriptsize{$\gamma_1$}}

\multiput(6.5,2.2)(-3,0){3}{\vector(-1,1){1.3}}
\multiput(6.5,2)(-3,0){3}{\vector(-2,-3){0.85}}
\multiput(6.45,2.1)(-1.5,0){6}{\vector(-1,0){1.2}}
\multiput(5.45,0.6)(-3,0){3}{\vector(-1,0){0.7}}
\multiput(4.5,0.7)(-3,0){3}{\vector(-2,3){0.85}}
\multiput(5,3.5)(-3,0){3}{\vector(-1,-1){1.3}}

\end{picture}
\end{center}
where we identify two vertices  denoted by $\bullet$ and $I$ is the ideal of $KQ$ generated by the elements $\gamma_3\gamma_2\gamma_1+\alpha_2\alpha_1+\beta_2\beta_1, \alpha_1\alpha_2, \beta_1\beta_2,\gamma_1\gamma_3$, $\alpha_2\alpha_1\beta_2\beta_1-\beta_2\beta_1\alpha_2\alpha_1$. Again, by \cite[Theorem 5.3]{KS1} $\Gamma_A$ has a decomposition
$$\Gamma_A=\mathcal{X}^A_0\vee \mathcal{C}^A_0 \vee \mathcal{X}^A_1\vee \mathcal{C}^A_1 \vee \mathcal{X}^A_2\vee \mathcal{C}^A_2$$
where, for each $0\leq i \leq 2$,
$\mathcal{C}^A_i=(\mathcal{C}^A_i(\lambda))_{\lambda\in\mathbb{P}_1(K)}$
is the $\mathbb{P}_1(K)$-family of quasi-tubes of $\Gamma_A$,
$\mathcal{X}^A_i$ is a family of components containing exactly one
simple module $S(i)$. Since $\varphi^3_{\widehat{B
}}=(\varphi^2_{\widehat{B}})\varphi_{\widehat{B}}=\nu_{\widehat{B}}\varphi_{\widehat{B}}$
and $\varphi_{\widehat{B}}$ is strictly positive automorphism of
$\widehat{B}$, by Theorem 1.1 we get that $\mathcal{C}^A_i$, for
some $0 \leq i \leq 2$, is generalized standard family of
quasi-tubes maximally saturated by simple and projective modules.
Further, applying Theorem 5.1 (see also \cite[Theorem 5.1]{KS1}),
we have that $\mathcal{C}^A_i$, for each $0\leq i\leq 2$, is a
canonical family of quasi-tubes of $\Gamma_{B^*_i}$, where $B^*_i$
is a quasi-tube enlargement of canonical algebra $C_i=C$. Note
that $B^*_i=KQ_{B^*_i}/I_{B^*_i}$, where $Q_{B^*_i}$ is of the
form \unitlength1cm
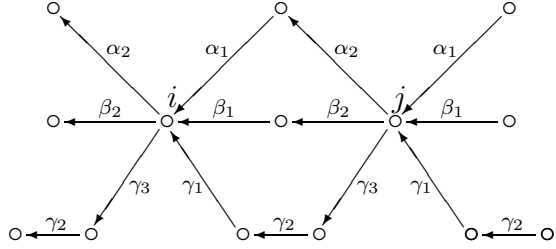
\begin{center}
\begin{picture}(3,4)

\multiput(5,2)(-1.5,0){5}{$\circ$}
\multiput(5,3.5)(-3,0){3}{$\circ$}
\multiput(5.5,0.5)(-1,0){2}{$\circ$}
\multiput(5.5,0.5)(-1,0){2}{$\circ$}
\multiput(2.5,0.5)(-1,0){2}{$\circ$}
\multiput(-0.5,0.5)(-1,0){2}{$\circ$}

\put(3.56,2.3){{$j$}} \put(0.6,2.3){{$i$}}

\multiput(2.8,3)(-3,0){2}{\scriptsize{$\alpha_2$}}
\multiput(4.1,3)(-3,0){2}{\scriptsize{$\alpha_1$}}
\multiput(2.7,2.2)(-3,0){2}{\scriptsize{$\beta_2$}}
\multiput(4.2,2.2)(-3,0){2}{\scriptsize{$\beta_1$}}
\multiput(5,0.7)(-3,0){3}{\scriptsize{$\gamma_2$}}
\multiput(3.8,1.2)(-3,0){2}{\scriptsize{$\gamma_1$}}
\multiput(3.1,1.2)(-3,0){2}{\scriptsize{$\gamma_3$}}

\multiput(3.5,2.2)(-3,0){2}{\vector(-1,1){1.3}}
\multiput(3.5,2)(-3,0){2}{\vector(-2,-3){0.85}}
\multiput(4.95,2.1)(-1.5,0){4}{\vector(-1,0){1.2}}
\multiput(5.45,0.6)(-3,0){3}{\vector(-1,0){0.7}}
\multiput(4.5,0.7)(-3,0){2}{\vector(-2,3){0.85}}
\multiput(5,3.5)(-3,0){2}{\vector(-1,-1){1.3}}
\end{picture}
\end{center}
and $I_{B^*_{i}}$ is the ideal of $KQ_{B^*_i}$ generated by
$\gamma_3\gamma_2\gamma_1+\alpha_2\alpha_1+\beta_2\beta_1, \alpha_1\alpha_2, \beta_1\beta_2,\gamma_1\gamma_3$ and $j\equiv i+1 (\mod 3)$. In conclusion, $\Gamma_A$ admits three generalized standard families $\mathcal{C}^A_i$, $0\leq i \leq 2$, of quasi-tubes maximally saturated by simple and projective modules.\\

\noindent{\scshape{Example 7.3.}} Let
$A=\widehat{B}/(\nu_{\widehat{B}})$ be a selfinjective algebra,
where $B=C$ is the canonical algebra as above. Then $A=KQ/I$ is
the bound quiver algebra, where $Q$ is the quiver \unitlength1.3cm
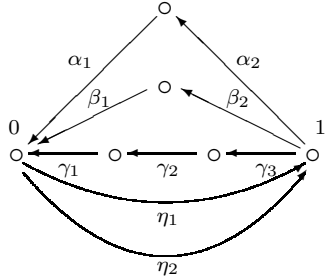
\begin{center}
\begin{picture}(3,2.5)
\multiput(3.5,1)(-1,0){4}{$\circ$}
\multiput(2,2.5)(-3,0){1}{$\circ$}
\multiput(2,1.7)(-1,0){1}{$\circ$}

\put(3.6,1.3){\scriptsize{$1$}}
\put(0.5,1.3){\scriptsize{$0$}}

\multiput(2.8,2)(-3,0){1}{\scriptsize{$\alpha_2$}}
\multiput(1.1,2)(-3,0){1}{\scriptsize{$\alpha_1$}}
\multiput(2.7,1.6)(-3,0){1}{\scriptsize{$\beta_2$}}
\multiput(1.3,1.6)(-3,0){1}{\scriptsize{$\beta_1$}}
\multiput(3,0.9)(-3,0){1}{\scriptsize{$\gamma_3$}}
\multiput(2,0.9)(-3,0){1}{\scriptsize{$\gamma_2$}}
\multiput(1,0.9)(-3,0){1}{\scriptsize{$\gamma_1$}}
\multiput(2,0.4)(-3,0){1}{\scriptsize{$\eta_1$}}
\multiput(2,-0.1)(-3,0){1}{\scriptsize{$\eta_2$}}

\multiput(3.5,1.2)(-3,0){1}{\vector(-1,1){1.3}}
\multiput(3.45,1.16)(-3,0){1}{\vector(-2,1){1.2}}
\multiput(1.9,1.75)(-3,0){1}{\vector(-2,-1){1.1}}

\multiput(3.4,1.1)(-1,0){3}{\vector(-1,0){0.7}}
\multiput(2,2.5)(-3,0){1}{\vector(-1,-1){1.3}}
\bezier{700}(0.65,1)(2.1,0.2)(3.5,1)
\bezier{700}(0.65,0.9)(2.1,-0.8)(3.5,0.9)
\put(3.4,0.93){\vector(3,2){0.1}}
\put(3.42,0.8){\vector(3,4){0.1}}
\end{picture}
\end{center}
and $I$ is the ideal of $KQ$ generated by
$\gamma_3\gamma_2\gamma_1+\alpha_2\alpha_1+\beta_2\beta_1,
\eta_1\alpha_2, \eta_2\beta_2, \alpha_1\eta_1, \beta_1\eta_2,
\eta_1\beta_2\beta_1-\eta_2\alpha_2\alpha_1,
\alpha_2\alpha_1\eta_2-\beta_2\beta_1\eta_1$. Following Theorem
5.1 and \cite[Theorem 5.3]{KS1}, $\Gamma_A$ has a decomposition
$$\Gamma_A=\mathcal{X}^A_0\vee \mathcal{C}^A_0 \vee
\mathcal{X}^A_1\vee \mathcal{C}^A_1$$ where, for $i \in\{0,1\}$,
$\mathcal{C}^A_i=(\mathcal{C}^A_i(\lambda))_{\lambda\in\mathbb{P}_1(K)}$
is the $\mathbb{P}_1(K)$-family of quasi-tubes of $\Gamma_A$,
$\mathcal{X}^A_i$ is a family of components containing exactly one
simple module $S(i)$. Further, $\mathcal{C}_0^A$ is the canonical
$\mathbb{P}_1(K)$-family of stable tubes of $\Gamma_C$ and
$\mathcal{C}_1^A$ is the canonical $\mathbb{P}_1(K)$-family of
quasi-tubes of $\Gamma_{B^*_1}$, where $B^*_1$ is the quasi-tube
enlargement of Kronecker algebra $C_1$. Note that
$\mathcal{C}_1^A$ contains projective $A$-module but no simple
$A$-modules (see \cite[Example 5.4]{KS1}). Since
$A=\widehat{B}/(\nu_{\widehat{B}})$ is symmetric algebra, there is
a nonzero homomorphism $f:P \rightarrow P$ for projective module
$P \in \mathcal{C}_1^A$ which factorizes through simple module
$\top P$. Hence $f \in \rad^{\infty}_A$ and only $\mathcal{C}_0^A$
satisfies the condition of Theorem 1.1.

\end{document}